%% file: paper1.tex
\documentclass[final,onefignum,onetabnum]{siamart190516}
\usepackage{geometry}                		% See geometry.pdf to learn the layout options. There are lots.
\usepackage{graphicx}				% Use pdf, png, jpg, or epsÂ§ with pdflatex; use eps in DVI mode
\usepackage{tikz}
\usetikzlibrary{automata,topaths,arrows,shapes}
\usetikzlibrary{backgrounds,patterns}
\usepackage{pgf}
\usepackage{graphicx}
\usepackage{ulem}
\usepackage{parskip}
\usepackage{caption}
\usepackage{pgfplots}
\usepackage{float}
\usepackage{mathrsfs}
\usepackage{hyperref}
\usepackage{diagbox}
\input{KM_definitions_7_2015}

\newcommand{\eps}{\varepsilon}

\newcommand{\red}[1]{#1}

\DeclareMathOperator{\Tr}{Tr}

\DeclareMathOperator{\length}{length}

\DeclareMathOperator{\sd}{sd} %singular directions

\newcommand{\til}{\widetilde}
\newcommand{\ol}{\overline}

\newcommand{\bs}{{\bf s}}
\newcommand{\SWITCH}{\mathsf{SWITCH}}

\newcommand{\LCC}{\mathsf{LCC}}

\newcommand{\In}{\mathsf{In}}

\newcommand{\off}{\mathsf{off}}
\newcommand{\on}{\mathsf{on}}

\newcommand{\parity}{\mathsf{par}}
\newcommand{\Eq}{\mathsf{Eq}}

\newsiamremark{example}{Example}

\usepackage{bbm}
\newcommand{\one}{\mathbbm{1}}
\newsiamremark{ex}{Example}

\title{Equilibria and their stability in networks with steep sigmoidal nonlinearities}
\author{William Duncan, Tomas Gedeon, Hiroshi Kokubu, Konstantin Mischaikow, and Hiroe Oka}
\date{\today}							% Activate to display a given date or no date

\begin{document}

\maketitle

 \begin{abstract}
 In this paper we investigate equilibria of continuous differential equation models of network dynamics. The motivation comes from gene regulatory networks where each directed edge represents either down- or up-regulation, and is modeled by a sigmoidal nonlinear function. We show that the  existence and stability of equilibria of a sigmoidal system is determined by a combinatorial analysis of the limiting switching system with piece-wise constant non-linearities. In addition, we describe  a local decomposition of a switching system  into a product of simpler cyclic feedback systems, where the cycles in each decomposition correspond to a particular subset of network loops.  
\end{abstract}

\section{Introduction}
Analysis of large systems of ordinary differential equations is difficult, especially when we seek to understand changes in dynamics when parameters  change. To set the stage, we are interested in systems of O(10) differential equations with the same order of magnitude of number of parameters; big enough to be complicated but not so large that statistical approaches may be applicable.
Systems of this size are important in systems biology, in particular, in models of gene regulation. Here variables usually represent concentrations of chemical species (mRNA, proteins) in the cell, and the interactions between variables are represented in a form of a network with signed directed edges. Nodes represent concentrations and directed edges monotone interactions; positive edges indicate activation and negative edges repression. The interactions are nonlinear; both on the level of the pairwise effect of  $x_i$ on $x_j$ which is usually modeled by a saturating function, but also on the level of how effects from different inputs combine together to influence $x_j$.  The choice of the nonlinearity that models  the effect of $x_i$ on $x_j$ is not given by any fundamental law of physics; the usual choices are Hill functions, but other sigmoid functions and threshold (switching) functions are used as well. 
This ambiguity, coupled with technical challenges related to simultaneous measurement of time evolution of multiple chemical species in a single cell, necessarily limits the expectation of fidelity of the model simulations with the experimental data. The model should not be expected to reproduce measurements in fine detail, but still answer qualitative questions on number and types stable equilibria are present, or capacity of the network to admit oscillations.
Taking into consideration ever present noise in molecular systems, there is always a need to address a question of how robust these qualitative features are under parameter changes in the model.

In this paper we concentrate on the existence and stability of equilibria in networks where the pairwise interaction is modeled by \emph{sigmoidal} nonlinearities. 
Extending results~\cite{snoussi93,veflingstad07,ironi11}, we show that  the equilibria in a network modeled by  sufficiently steep sigmoidal functions are in one-to-one correspondence with a collection of so called regular and singular equilibria of a model of the same network using  \emph{switching} functions.
Switching functions are  piece-wise constant functions with a single threshold
and range with two values $\{L,U\}$, which can be interpreted as two rates  of expression (L for Lower, and U for Upper) of the target gene based on whether the controlling gene is below, or above the threshold $\theta$. These models have been used for  gene regulatory networks  since the 70's~\cite{glass:kauffman:73, Glass1978, deJong2002,Thomas1991, edwards00, cummins16,ironi11,Gedeon2020}. However, using these functions as the right hand side of an ODE system presents several technical challenges, especially how to deal with the fact that the  vector field is not defined at thresholds $\theta$. \red{One approach to extend switching systems so that they are defined at thresholds is to view them as differential inclusions rather than ordinary differential equations. Stability of singular equilibria of the switching systems from this point of view were studied by \cite{casey06}.
Convergence of switching systems dynamics was studied also by \cite{pasquini20} who formulated conditions that are needed to construct a global Lyapunov function for a switching system by piecing together local Lyapunov functions. }  The idea of the new DSGRN (Dynamic Signatures Generated by Regulatory Networks) approach~\cite{cummins16,Gedeon18,Gedeon2020}, supported by a suite of corresponding software~\cite{DSGRNgithub}, is to capture information about the network dynamics given by switching system models in a form of combinatorial (finite) data,  and then use this data to rigorously establish results about well-defined dynamics of ODE's with continuous right hand side that are a small perturbation of the switching functions. \red{We emphasize that the distinction in the two approaches is that the Filippov extension approach is concerned with the dynamics of switching systems whereas the DSGRN approach only uses switching systems as a computational tool for the study of the dynamics of sigmoidal systems. }

While we describe the combinatorial data in greater detail below, for the purpose of this introduction it is sufficient to note that the switching system ODE only contains stable equilibria (which we call regular), because the unstable equilibria of sigmoidal systems limit to intersections of thresholds of switching systems. In this paper we show that the intersections,
the intersections that appear as such limits of unstable equilibria, which we call \emph{loop characteristic equilibrium cells}, can be precisely characterized using combinatorics of the switching system. In particular, we add to the DSGRN approach by showing how to use the combinatorial data from a switching system to predict existence and stability of all equilibria for all sigmoidal functions that \red{arise as perturbations} to the switching functions. The combinatorial data only uses the type (i.e. positive feedback or negative feedback) and the number of feedback loops in the network. To obtain these results we show that switching systems can be locally (in phase space) decomposed into a product of simpler cyclic feedback systems.  Each such cyclic system can be associated to a unique oriented loop in the gene regulatory network.

 A potential  application of description of  equilibria and their stability in sigmoidal systems   is in the  recurrent artificial neural networks (rANN). These models were  introduced by Hopfield~\cite{Hopfield1982} and Grossberg~\cite{Grossberg1988} almost 40 years ago, but they found their newest incarnation as Echo state networks~\cite{Jaeger2004}. There is a great variety of implementations but at the core there are  network nodes (i.e. neurons) that are connected by weighted directed edges. 
 Each node processes the input through a  nonlinear function (binary, sigmoidal, or a ramp).   Our work provides a  characterization of the number of stable equilibria for steep sigmoidal functions baased on the combinatorics of the switching system, which is ultimately tied to the structure of the connections in the network.

\subsection{Organization of the paper}
In Section \ref{sec: RN}, we define sigmoidal functions and switching functions. In Section~\ref{sec: eq cells}, we define the combinatorial data associated to switching functions and show how this data can be used to identify all equilibria of steep sigmoidal systems. 
The proofs for these results are in Section~\ref{sec: eq cell proofs}. In Section~\ref{sec: cfs results}, we use the results of Section \ref{sec: eq cells} to characterize the equilibria of cyclic feedback networks and then analyze their stability. 
The proofs can be found in Section~\ref{sec: cfs results proofs}. In Section~\ref{sec: gen sys results}, we show that all switching systems can be locally (in phase space) decomposed as a product of cyclic feedback systems and use this decomposition to generalize the stability results of Section~\ref{sec: cfs results} from cyclic feedback systems to general networks. Finally, in Section~\ref{sec: discussion}, we conclude with a discussion of our results.

\section{The Regulatory Network and Switching Systems}\label{sec: RN}
\begin{definition}[\cite{cummins16}]
A \emph{regulatory network} $\bRN = (V,E)$ is an annotated finite directed graph with vertices $V = \{1,\ldots, N\}$ called \emph{network nodes} and directed edges $E\subset V\times V \times \{1,-1\}$. 
 An annotated edge  $(j, i,+1)$ represents an \emph{activation} of node $i$ by  node $j$ and is denoted $j\to i$;  annotated edge   $(j, i,-1)$ represents \emph{repression} of node $i$ by  node $j$ and is denoted $j\dashv i$. 
We write $\bs_{ij}=1$ if $j\to i$ and $\bs_{ij} = -1$ if $j\dashv i$. We indicate either $j\to i$ or $j\dashv i$ without specifying which by writing $(j,i)\in E$. We allow self edges, but admit at most one edge between any two nodes. The set of \emph{sources} and \emph{targets} of a node are denoted by 
\[\Sources(k) = \{j\;|\; (j,k)\in E\} \qquad \mbox{and} \qquad \Targets(k) = \{j\;|\; (k,j)\in E\}\]  
and we require every node has a target. 
 \end{definition}
 
 We remark that the assumption that every node has a target is not a serious constraint. If a node $j$ does not have a target, then the dynamics of the remaining nodes are independent of $j$. Once the dynamics of the remaining nodes are understood, the dynamics of node $j$ can be treated as a non-autonomous system driven by the remaining nodes. 
 
 To an $\bRN$ we associate a \emph{switching system} of the form
 \begin{align}\label{eq: switch sys}
  \dot{x} = -\Gamma x + \Lambda(x)
 \end{align}
 where $\Gamma$ is a diagonal matrix with entries $\Gamma_{jj} = \gamma_j$ and $\Lambda$ is a nonlinear function of the form 
 \begin{align}\label{eq: Lambda}
  \Lambda_i(x) := \prod_{\ell=1}^{p_i}\sum_{j\in I_\ell} \sigma_{ij}(x_j)
 \end{align}
 with $I_1,\ldots,I_{p_i}$ a partition of $\Sources(i)$. Each $\sigma_{ij}$ is a \emph{switching function} of the form
\begin{align}\label{eq: sigma}
 \sigma_{ij}(x_j) := \begin{cases}
                     L_{ij}, \quad & \bs_{ij} = 1 \mbox{ and } x_j<\theta_{ij} \mbox{ or } \bs_{ij}=-1 \mbox{ and } x_j>\theta_{ij} \\
                     U_{ij}, \quad & \bs_{ij}=1 \mbox{ and } x_j>\theta_{ij} \mbox{ or } \bs_{ij}=-1 \mbox{ and } x_j<\theta_{ij} \\
                     \mbox{undefined}, \quad & \mbox{ if } x_j=\theta_{ij}.
                    \end{cases}
\end{align} 
\newcommand{\switch}{\mathsf{SWITCH}}
The parameter $Z= (L,U,\theta,\Gamma)$, where $L:=(L_{ij})$, $U:=(U_{ij})$,  $\theta:= (\theta_{ij})$ are vectors indexed by $(ij)$, is the \emph{switching parameter}. We denote a switching system parameterized by $Z$ by $\SWITCH(Z)$.

To an $\bRN$ we also associate a \emph{sigmoidal system}, $\cS(Z,\eps)$, where $Z$ is a switching parameter and $\eps\in \bR^{N\times N}$ is a \emph{perturbation parameter}. We say $\eps'\leq \eps$ or $\eps'<\eps$ when the component-wise comparisons $\eps_{ij}'\leq \eps_{ij}$ or $\eps_{ij}'<\eps_{ij}$ hold for each $(j,i)\in E$, respectively. The pair $(Z,\eps)$ is the \emph{sigmoidal parameter}. The dynamics of $\cS(Z,\eps)$ are given by 
\begin{align}\label{eq: sigmoid sys}
    \dot{x} = -\Gamma x + \Lambda(x;\eps)
\end{align}
where $\Lambda(x;\eps)$ is obtained from $\Lambda$ by replacing the switching functions $\sigma_{ij}$ with \emph{sigmoidal perturbations} $\sigma_{ij}(\cdot;\eps_{ij})$, which we define below.

\begin{definition}\label{defn: sigmoid}
    $\sigma_{ij}(\cdot;\eps_{ij})$ is a family of \emph{sigmoidal perturbations} of $\sigma_{ij}$ at a parameter $Z$ if for each $\eps_{ij}\in \bR_+$, 
    \begin{enumerate}
        \item $\sigma_{ij}(\cdot;\eps_{ij})$ is continuously differentiable and monotone non-increasing or monotone non-decreasing,
        \item $\sup_x\,\sigma_{ij}(x;\eps_{ij}) = U_{ij}$ and $\inf_x\,\sigma_{ij}(x;\eps_{ij}) = L_{ij}$,
        \item There is a neighborhood $\cU_1(\eps_{ij})\subset\bR$ of $\theta_{ij}$ such that $\cU_1(\eps_{ij}) \to \{\theta_{ij}\}$ as $\eps_{ij}\to 0$ and a constant $C_1>0$, such that if $x\in \bR\setminus\cU_1(\eps_{ij})$, then $|\sigma_{ij}'(x;\eps_{ij})| \leq C_1 \eps_{ij}$.  
        \item There is a neighborhood $\cU_2(\eps_{ij}) \subset \bR$ of $\theta_{ij}$ such that $\cU_2(\eps_{ij}))\to \{\theta_{ij}\}$ \red{and $\sigma_{ij}(\cU_2(\eps_{ij});\eps_{ij}) \to (L_{ij},U_{ij})$} as $\eps_{ij}\to 0$ and a constant $C_2>0$ such that if $x\in U_2(\eps_{ij})$ then $|\sigma_{ij}'(x;\eps_{ij})| \geq C_2 \eps_{ij}^{-1}$.
    \end{enumerate}
\end{definition}

 Given a perturbation parameter $\eps$, we will write $\sigma_{ij}(\cdot;\eps)$ instead of  $\sigma_{ij}(\cdot;\eps_{ij})$ to simplify notation. Note that \red{as $\eps \to 0$, the sigmoidal perturbation $\sigma_{ij}(\cdot;\eps)$ converges pointwise to the step function $\sigma_{ij}$. }
 %\sout{$\cS(Z,0) = \SWITCH(Z)$.}} 
 Given a switching parameter $Z$ and perturbation parameter $\eps\geq 0$, we denote the Jacobian of \eqref{eq: sigmoid sys} at $x$ by $J(x;\eps)$ or $J(\eps)$ when $x$ is implied from context.

\begin{example}\label{ex: positive toggle}
Throughout the paper we will illustrate the concepts on a simple example of a two node network we call the \emph{positive toggle plus}, where two nodes activate themselves and mutually activate each other, i.e.
\[
    \bRN = (V,E) = (\{1,2\}, \; \{(1\to 1), (2\to 2), (1\to 2), (2\to 1)\}). 
\]
The name "positive toggle" refers to the network without self loops and was chosen for its resemblance to the toggle switch introduced in \cite{gardner00}, in which the nodes mutually repress each other rather than activate. The "plus" modifier refers to the addition of the self loops. The associated switching system has the form
\begin{align*}
    \dot{x}_1 &= -\gamma_1 x_1 + \sigma_{11}(x_1)\sigma_{12}(x_2) \\
     \dot{x}_2 &= -\gamma_2 x_2 + \sigma_{22}(x_2)\sigma_{21}(x_1)
\end{align*}
and we note $s_{11} = s_{22} = s_{12}=s_{21} = 1$. \red{We will consider this system with a switching parameter satisfying
\begin{align*}
    &L_{11}L_{12} < L_{11}U_{12} < \gamma_1\theta_{21} < U_{11}L_{12} < \gamma_1\theta_{11} < U_{11}U_{12}, \mbox{ and }\\
    &L_{22}L_{21}< \gamma_2\theta_{12} < L_{22}U_{21} < U_{22}L_{21} < U_{22}U_{21}<\gamma_2\theta_{22}.
\end{align*}
}
\end{example}

\section{Equilibria of Regulatory Networks}\label{sec: eq cells}

 As observed in \cite{cummins16}, the thresholds $\theta_{ij}$ of a switching system impose a grid-like structure on phase space $\bR^N_+$. In this section we characterize where the equilibria lie in phase space relative to this structure. We begin by defining the structure.

\begin{definition}\label{defn: cell complex}
\begin{enumerate}
\item For each $j\in V$, we define $\theta_{-\infty j} := 0$, $\theta_{\infty j} := \infty$, and 
\[
\Theta_j(Z) := \{\theta_{ij} > 0\;|\; i\in \Targets(j)\}\cup \{\theta_{\infty j}, \theta_{-\infty j}\}.
\] 
The \emph{threshold set} is the collection $\Theta(Z) := (\Theta_1(Z),\ldots,\Theta_N(Z))$. We say $\theta_{i_1 j},\theta_{i_2 j} \in \Theta_j(Z)$ are \emph{consecutive thresholds} if $\theta_{i_1 j} < \theta_{i_2 j}$ and there does not exist $\theta_{i_3 j}\in \Theta_j(Z)$ such that $\theta_{i_1 j}<\theta_{i_3 j}<\theta_{i_2 j}$. 

\item A \emph{cell}, $\tau$ associated to a threshold set $\Theta$, is a product of $k\leq N$ thresholds and $N-k$ open intervals whose endpoints are consecutive thresholds. By renumbering the variables we write
\begin{align*}
\tau = \prod_{j=1}^k \{\theta_{i_j j}\} \times \prod_{j=k+1}^N (\theta_{a_j j},\theta_{b_j j}).
\end{align*}
We write $\pi_j(\tau)$ for the projection of $\tau$ onto the $j$th direction. A cell is \emph{regular} if $k=0$ and \emph{singular} otherwise. The \emph{cell complex}, $\chi(\Theta)$, is the collection of all cells associated to the threshold set $\Theta$. The cell complex associated to the switching system \ref{eq: switch sys} at parameter $Z$ is $\chi(\Theta(Z))$. When the switching parameter $Z$ and the threshold set $\Theta(Z)$ are clear from context we drop the argument $\Theta(Z)$ and write $\chi$.

 Figure~\ref{fig: neighbors}(a) depicts the cell complex $\chi$ for the positive toggle plus. The concept of neighboring cells is described below.
 \end{enumerate}
\end{definition}

\def\psz{1mm}
\pgfdeclarepatternformonly{dashed vertical}{\pgfpointorigin}{\pgfpoint{\psz}{\psz}}{\pgfpoint{1.5*\psz}{1.5*\psz}}{
\pgfsetlinewidth{.15mm}
\pgfpathmoveto{\pgfpoint{.5*\psz}{0mm}}
\pgfpathlineto{\pgfpoint{.5*\psz}{\psz}}
\pgfpathclose
\pgfusepath{stroke}
}

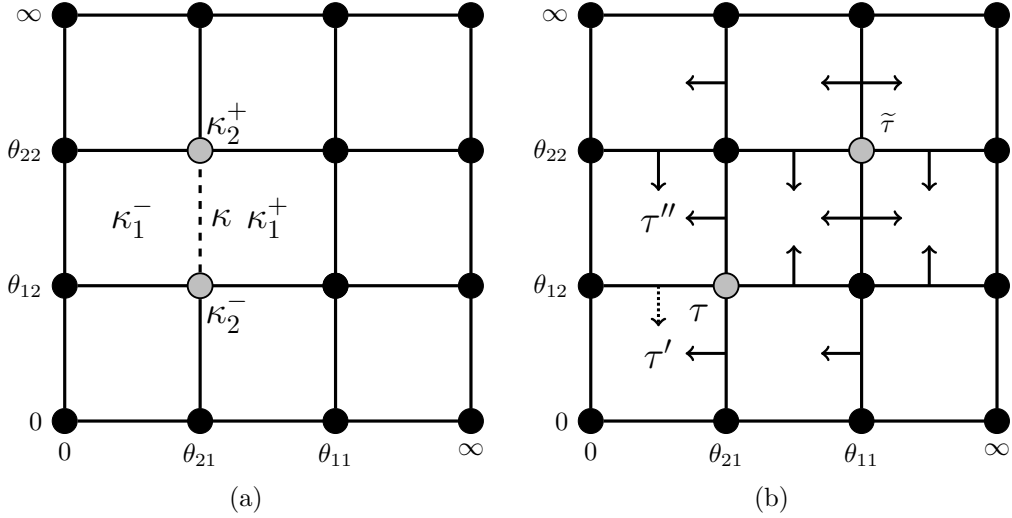
\begin{figure}[htbp!]
\begin{tabular}{cc}
  \begin{tikzpicture}[scale = 1.8]
    \tikzstyle{line} = [-,very thick]
    \tikzstyle{arrow} = [->,line width = .4mm]
    \tikzstyle{unstable} = [red]
    \tikzstyle{stable} = [blue]
    \tikzstyle{pt} = [circle,draw=black,fill = black,minimum size = 2pt];
    \tikzstyle{labeled pt} = [circle, draw = black, fill = gray!50!white, minimum size = 2pt, line width = .25mm];
    \tikzstyle{cone pt} = [circle, draw = black, minimum size = 1pt, fill = blue];
    \def\epsshift{.3}
      
    %vertices of cells
    \node[pt] at (0,0) (00) [label =below: 0,label = west: 0] {};
    % labeled cell
    \node[labeled pt] at (1,1) (11) {};
    \node at (1.2,.8) {\Large $\kappa_2^-$};
    
    \node[pt] at (1,0) (10) [label = below:$\theta_{21}$ ] {};
    \node[pt] at (2,0) (20) [label = below:$\theta_{11}$ ] {};

    \node[pt] at (2,1) (21) {};
    \node[pt] at (2,2) (22) {};
    \node[pt] at (2,1) (21) {};
    %labeled cell
    \node[labeled pt] at (1,2) (12) {};
    \node at (1.2,2.2) {\Large $\kappa_2^+$};
    
    \node[pt] at (3,0) (30) [label = below: $\infty$] {};
    \node[pt] at (3,1) (31) {};
    \node[pt] at (3,2) (32) {};

    \node[pt] at (0,1) (01)  [label = west: $\theta_{12}$] {};
    \node[pt] at (0,2) (02)  [label = west: $\theta_{22}$] {};
    \node[pt] at (0,3) (03) [label = west: $\infty$] {};
    \node[pt] at (1,3) (13) {};
    \node[pt] at (2,3) (23) {};
    \node[pt] at (3,3) (33) {};
    
    %labeled cells
    \node at (.5,1.5) {\Large $\kappa_1^-$};
    \node at (1.5,1.5) {\Large $\kappa_1^+$};

    %draw faces of cells
    \begin{pgfonlayer}{background}
    \draw[line] (00) to (30);
    \draw[line] (00) to (03);
    \draw[line] (10) to (11);
    \draw[line] (12) to (13);
    \draw[line] (20) to (23);
    \draw[line] (30) to (33);
    
    \draw[line] (01) to (31);
    \draw[line] (02) to (32);
    \draw[line] (03) to (33);
    %\draw[fill = gray!25!white] (01) rectangle (22);
    \draw[line, dashed] (11) to node [right] {\Large $\kappa$} (12);
    \end{pgfonlayer}
  \end{tikzpicture} 
  &
  \begin{tikzpicture}[scale = 1.8]
    \tikzstyle{line} = [-,very thick]
    \tikzstyle{arrow} = [->,line width = .4mm]
    \tikzstyle{unstable} = [red]
    \tikzstyle{stable} = [blue]
    \tikzstyle{pt} = [circle,draw=black,fill = black,minimum size = 2pt];
    \tikzstyle{cone pt} = [circle, draw = black, minimum size = 1pt, fill = blue];
    \tikzstyle{labeled pt} = [circle, draw = black, fill = gray!50!white, minimum size = 2pt, line width = .25mm];
    \def\epsshift{.3}
    \def\arrlen{.3}

    %regular equilibrium cells
    % \filldraw[fill = blue!40!white] (0,2) rectangle (1,3);
    % \filldraw[fill = blue!40!white] (1,0) rectangle (2,1);
      
    %vertices of cells
    \node[pt] at (0,0) (00) [label =below: 0,label = west: 0] {};
    
    \node[pt] at (1,0) (10) [label = below:$\theta_{21}$ ] {};
    \node[pt] at (2,0) (20) [label = below:$\theta_{11}$ ] {};
    
    \node[labeled pt] at (1,1) (11) {};
    \node[pt] at (2,1) (21) {};
    \node[pt, labeled pt] at (2,2) (22) [label = 45: $\til{\tau}$] {};
    \node[pt] at (2,1) (21) {};
    \node[pt] at (1,2) (12) {};
    \node[pt] at (3,0) (30) [label = below: $\infty$] {};
    \node[pt] at (3,1) (31) {};
    \node[pt] at (3,2) (32) {};

    \node[pt] at (0,1) (01)  [label = west: $\theta_{12}$] {};
    \node[pt] at (0,2) (02)  [label = west: $\theta_{22}$] {};
    \node[pt] at (0,3) (03) [label = west: $\infty$] {};
    \node[pt] at (1,3) (13) {};
    \node[pt] at (2,3) (23) {};
    \node[pt] at (3,3) (33) {};
    
    %labeled cells
    \node at (.8,.8) {\Large $\tau$};
    \node at (.5,.5) {\Large $\tau'$};
    \node at (.5,1.5) {\Large $\tau''$};
    
    %draw faces of cells
    \begin{pgfonlayer}{background}
    \draw[line] (00) to (30);
    \draw[line] (00) to (03);
    \draw[line] (10) to (13);
    \draw[line] (20) to (23);
    \draw[line] (30) to (33);
    
    \draw[line] (01) to (31);
    \draw[line] (02) to (32);
    \draw[line] (03) to (33);
    \end{pgfonlayer}

    %arrows
    \draw[arrow,densely dotted] (.5,1) to (.5,1-\arrlen);
    \draw[arrow] (1,.5) to (1-\arrlen,.5);
    \draw[arrow] (1.5,1) to (1.5,1+\arrlen);
    \draw[arrow] (1,1.5) to (1-\arrlen,1.5);
    \draw[arrow] (.5,2) to (.5,2-\arrlen);
    \draw[arrow] (2,.5) to (2-\arrlen,.5);
    \draw[arrow] (1.5,2) to (1.5,2-\arrlen);
    \draw[arrow] (2.5,2) to (2.5,2-\arrlen);
    \draw[arrow] (2.5,1) to (2.5,1+\arrlen);
    \draw[arrow] (1,2.5) to (1-\arrlen,2.5);
    \draw[arrow] (2,2.5) to (2-\arrlen,2.5);
    \draw[arrow] (2,2.5) to (2+\arrlen,2.5);
    \draw[arrow] (2,1.5) to (2-\arrlen,1.5);
    \draw[arrow] (2,1.5) to (2+\arrlen,1.5);

  \end{tikzpicture} \\
  (a) & (b)
\end{tabular}
\caption{ \textbf{The cell complex $\chi$, neighbors, and the labeling map in the positive toggle plus network.} \red{The network and switching system are defined in Example \ref{ex: positive toggle}.}
%\sout{The chosen switching parameter $Z$ satisfies $L_{11}L_{12} < \theta_{21} < L_{11}U_{12} < U_{11}L_{12} < \theta_{11} < U_{11}U_{12}$ and $L_{22}L_{21}< \theta_{12} < L_{22}U_{21} < U_{22}L_{21} < U_{22}U_{21}<\theta_{22}$.} }
\textbf{(a):} Each box, line, and point is a cell in the cell complex $\chi$.  The cell $\kappa=\{\theta_{21}\}\times(\theta_{12},\theta_{22})$ is indicated by the dashed line. For $\kappa$, the direction $1$ is  singular and the direction $2$ is regular. \red{ The 2-neighbors of $\kappa$ (see Definition \ref{defn: neighbors}),} $\kappa_2^+$ and $\kappa_2^-$, are indicated by the gray circles while  \red{the 1-neighbors,} $\kappa_1^-$ and $\kappa_1^+$, are the labeled  two dimensional cells. \textbf{(b):} The loop characteristic cell\red{s} $\tau$ \red{and $\til{\tau}$ are} indicated by the gray circle\red{s} and the regular cells $\tau'$ \red{and $\tau''$ are} labeled. The dotted arrow represents \red{$\cL(\tau,1,-)$} for which $1$ is a singular direction and, at the same time, it represents \red{$\cL(\tau',2,+)$} \red{and $\cL(\tau'',2,-)$} for which 2 is a regular direction. \red{See Example \ref{ex: labeling map} for details}. These values are equal to $-1$ so the arrow points down. The arrows on the outside cell complex  boundary are not drawn because they point inwards for all choices of parameters. The arrows of the labeling map imply that $\Phi_1(\tau) = \Phi_2(\tau') = \red{0}$. \red{The arrows on the top and bottom boundaries of $\tau''$ indicate that $\Phi_2(\tau'') = -1$. }  }\label{fig: neighbors}
\end{figure}

 Observe, that in the switching system, the function $\Lambda$
is only defined on regular cells and not on singular cells. Therefore, equilibria of $\SWITCH(Z)$ can only be contained in regular cells. \red{ Equilibria in singular cells can be understood if \eqref{eq: switch sys} is replaced with its Filippov extension wherein the differential equation is replaced by a differential inclusion. This was done in \cite{casey06} where existence and stability of these singular equilibria was studied. However, our goal is to understand equilibria of the sigmoidal systems $\cS(Z,\eps)$ which are perturbations of the switching system $\SWITCH(Z)$, and not the equilibria of $\SWITCH(Z)$. For this reason, in the next definition we define an equilibrium cell to be a cell which a sigmoidal equilibrium converges to as $\eps\to 0$. It is straightforward to see that regular equilibrium cells contain a unique stable equilibrium of $\SWITCH(Z)$.
 }  
%  \red{\sout{Furthermore, since  $\Lambda$ is constant on any regular cell, the Jacobian $J(x,0)$ on a regular cell is a diagonal matrix with entries $-\gamma_i$ and these equilibria are stable. Therefore $\SWITCH(Z)$ only contains stable equilibria and they lie in regular cells.
% However, a sigmoidal system $\cS(Z,\eps)$ may have unstable equilibria for arbitrary small  $\eps$.  Since $\cS(Z,\eps)\to \SWITCH(Z)$ as $\eps \to 0$
% any unstable equilibrium must converge to a singular cell.
%  We will call such singular cells, as well as regular cells that contain stable equilibria of $\SWITCH(Z)$, \emph{equilibrium cells}. These cells are explicitly defined below. }}

\begin{definition}
  Let $Z$ be a switching parameter and $\tau\in\chi$. If there is an $A\in \bR_+^{N\times N}$ so that for all $\eps< A$, $\cS(Z,\eps)$ has a fixed point $x^{\eps}$ satisfying $x^{\eps} \to \tau$ as $\eps\to 0$, then $\tau$ is an \emph{equilibrium cell}. If $\tau$ is a singular cell, then $x^{\eps}$ is a \emph{singular stationary point} (SSP).
 \end{definition}

 Theorem \ref{thm: FP} characterizes equilibrium cells \red{using combinatorial information about $\SWITCH(Z)$ only. That is, information about the switching system is necessary and sufficient to characterize the equilibria of the sigmoidal system $\cS(Z,\eps)$ when $\eps$ is small enough.} To state the theorem precisely, we need some additional definitions to describe the cell complex and a notion of a flow direction map on the cell complex.  
 
 \subsection{Cell Complex}

\begin{definition}

 \begin{enumerate}

  \item Given $\tau\in \chi$, the \emph{singular directions} of $\tau$, denoted $\sd(\tau)$, correspond to the set of indices, $s$, such that $\pi_s(\tau) = \{\theta_{i_s s}\}$. An index is a \emph{regular direction} if it is not singular. We define $\rho^\tau:V\to V$ by 
 \begin{align*}
  \rho^\tau(j) := \begin{cases}
                  i_j, \quad & j\in \sd(\tau)\\
                  j, \quad &\mbox{ otherwise}. 
                 \end{cases}
 \end{align*}
The set of cells with $k$ singular directions is denoted $\chi^{(N-k)}$. 
The map $\rho^\tau$ depends on the cell $\tau$, but when  $\tau$ is fixed and clear from the context, we will use $\rho$ instead of $\rho^\tau$.

\item  If $s$ is a singular direction of $\tau\in\chi$ we denote the \emph{neighboring thresholds} by $\theta_{\rho_{-}(s) s}$ and $\theta_{\rho_{+}(s) s}$ where $\theta_{\rho_{-}(s) s}<\theta_{\rho(s) s}< \theta_{\rho_{+}(s) s}$ are consecutive thresholds in $\Theta_s(Z)$. If $r$ is a regular direction of $\tau$ we write $\pi_r(\tau) = (\theta_{a_r r},\theta_{b_r r})$.
 \end{enumerate}
\end{definition}

\begin{definition}
  A cell, $\tau\in \chi$, is a \emph{loop characteristic cell} if $\rho^\tau$ is a permutation on $\sd(\tau)$. We denote the set of loop characteristic cells by $\LCC$. Note that all $N$-dimensional cells $\kappa$ are automatically loop characteristic cells, since $\sd(\kappa) = \emptyset$. Therefore $\chi^{(N)}\subset \LCC$.
\end{definition}

\red{A loop characteristic cell is a cell in which some number of disjoint loops of the network are active. For example, the loop characteristic cells of the positive toggle plus network (see Example \ref{ex: positive toggle}) include $\{\theta_{21}\}\times \{\theta_{12}\}$, where the loop $1\to2\to1$ is active, $\{\theta_{11}\}\times (\theta_{12},\theta_{22})$, where the loop $1\to 1$ is active, and $\theta_{11}\times \theta_{22}$, where the loops $1\to 1$ and $2\to 2$ are both active. The concept of loop characteristics was introduced in \cite{snoussi93} where it was shown singular equilibria of switching systems are contained in loop characteristic cells.  %where it was shown that singular equilibria of a switching system extended so that singular states are well defined must be contained in singular loop characteristic cells.  This
This was later  extended to sigmoidal systems by \cite{veflingstad07} who showed} that equilibrium cells are a subset of loop characteristic cells \red{when the sigmoidal perturbations are taken to be Hill functions}. \red{Theorem \ref{thm: FP} extends this work by considering a more general class of sigmoids and providing necessary and sufficient conditions for a loop characteristic cell to be an equilibrium cell.}

\red{
In the following definition we introduce notation to describe the \emph{neighbor} of a cell. By a neighbor to $\tau$ we mean a cell which is directly adjacent to $\tau$. 
}
\begin{definition}\label{defn: neighbors}
 Let $\tau\in \chi$ and $j\in V$.The \emph{left $j$-neighbor} (see Figure \ref{fig: neighbors}) of the cell $\tau$ is a cell  $\tau_j^-$, defined by 
 \begin{align*}
     \pi_k(\tau_j^-) := \begin{cases}
        \pi_j(\tau), \quad & j\neq k\\
        (\theta_{\rho_-(j)j},\theta_{\rho(j)j}), \quad & j=k,\; k\in\sd(\tau)\\
        \{\theta_{a_j j}\}, \quad & j=k,\; k\notin\sd(\tau). 
     \end{cases}
 \end{align*}
 Similarly, the \emph{right $j$-neighbor}, $\tau_j^+$, is defined by 
 \begin{align*}
     \pi_k(\tau_j^+) := \begin{cases}
       \pi_k(\tau), \quad & j\neq k\\
       (\theta_{\rho(j)j},\theta_{\rho_+(j)j}), \quad & j=k,\; k\in \sd(\tau) \\
       \{\theta_{b_jj}\},\quad & j=k,\; k\notin\sd(\tau).
     \end{cases}
 \end{align*}
 An \emph{$j$-neighbor} of $\tau$ is either a left or right $j$-neighbor of $\tau$. A \emph{neighbor} of $\tau$ is any $j$-neighbor.
\end{definition}
\red{
On a diagram of the cell complex, the left $j$-neighbor of $\tau$, $\tau_j^-$, is the cell directly below $\tau$ in the $j$th direction and the right $j$-neighbor of $\tau$, $\tau_j^+$, is the cell directly above $\tau$ in the $j$th direction. If $j$ is a singular direction of $\tau$ then $j$ is a regular direction of $\tau_j^\pm$. If $j$ is a regular direction of $\tau$ then $j$ is a singular direction of $\tau_j^\pm$. See Figure \ref{fig: neighbors}(a) for an example. 
}

\red{In the remainder of the paper we assume that the thresholds $\theta_{ij}$ are positive. This reflects that our motivation comes from biological networks in which activities or concentrations of a reactant are always non-negative. We also assume that thresholds corresponding to the same node are not equal. This holds generically, greatly simplifies our analysis, and is typical in the literature on switching systems \cite{snoussi93,casey06,farcot09,veflingstad07,ironi11,cummins16,pasquini20}. These assumptions are captured in the following definition.}

\begin{definition}
 The switching parameter $Z$ is \emph{threshold regular} if 
 \begin{itemize}
    \item For all $(j,i)\in E$, $\theta_{ij}>0$, and 
     \item for all $j\in V$, $i_1,i_2\in \Targets(j)$, $\theta_{i_1 j} \neq \theta_{i_2 j}$. 
 \end{itemize}
\end{definition}

\begin{definition}[\red{Definition 4.6 of} \cite{cummins16}]
 Consider a threshold regular switching parameter $Z$. For $j\in V$, denote the ordering of the thresholds $\{\theta_{ij}\;|\; i\in \Targets(j)\}$ by $O_j(Z)$. The \emph{order parameter} of $Z$ is the collection of these orders, $O(Z) = (O_1(Z),\ldots,O_N(Z))$.
\end{definition}

\subsection{Flow Direction Map}

 The goal of this section is to define a  flow direction map on the cell complex $\chi(\Theta(Z))$, which is induced by the right hand side of the switching system (\ref{eq: switch sys}).  We start by introducing some notation. Observe  that  if $j$ is a singular direction 
of  $\tau\in\chi$ and thus $\pi_j(\tau) = \{\theta_{ij}\}$ for some $i$, then by equation \eqref{eq: sigma} the function $\sigma_{ij}(x_j)$ is not defined on $\tau$. However, if $j$ is a regular direction of $\tau$ and thus $\pi_j(\tau) = (\theta_{i_1j}, \theta_{i_2j})$, then $\sigma_{ij}(x_j)$ is constant on $\tau$. We denote its value by $\sigma_{ij}(\tau)$. It follows that for $\tau \in\chi^{(N)}$ which has no singular directions $\sd(\tau) = \emptyset$, the value of $\Lambda_i(\tau)$ is well defined and constant for very $i$. \red{The vector $\Gamma^{-1}\Lambda(\tau)$ is sometimes called the focal point of $\tau$ because any trajectory of the switching system with initial condition in $\tau$ will converge to this value until the trajectory reaches the boundary of $\tau$ \cite{casey06}. In the following definition we give non-degeneracy conditions for the switching parameter $Z$ which we will assume throughout. } %\red{\sout{We use this notation to define a non-degeneracy condition on the switching parameter $Z$.}}

\begin{definition}[\red{Definition 2.7 of} \cite{cummins16}]
 The switching parameter $Z$ is \emph{regular} if 
 \begin{itemize}
  \item $Z$ is threshold regular,
  \item for all $(j,i)\in E$, $0<L_{ij}<U_{ij}$,
  \item for all $k\in V$, $\gamma_k>0$, and
  \item for all $\kappa\in \chi^{(N)}$ and $(j,i)\in E$, $\gamma_j\theta_{ij} \neq \Lambda_j(\kappa)$ for each threshold $\theta_{ij}$ which defines $\kappa$. 
 \end{itemize}
\end{definition}

%\red{\sout{Note that for the regular switching parameter $Z$,}} 
\red{The last condition expresses the requirement that the focal point of each regular cell $\kappa$ does not lie in the boundary of $\kappa$, which holds generically.  This is a typical assumption for switching systems because it implies that the} right hand side of the switching system \eqref{eq: switch sys} has a well defined crossing direction on all the boundaries of cells $\kappa \in \chi^{(N)}$ \red{\cite{casey06,farcot09,cummins16,pasquini20}}. We will use this to define a labeling map that collects information about these crossing directions. We then use the labeling map to define the flow direction map which indicates the direction in which the flow of the corresponding system crosses the the threshold.  The flow direction map can be viewed as  a multi-valued map  and  represented as a state transition graph that is a combinatorial summary of the flow information given by the switching system.
%\sout{and ramp systems, respectively.}} 
The \emph{labeling map}, defined below, generalizes the concept of wall-labeling (Definition 3.1 of \cite{cummins16}) in switching systems from regular cells to loop characteristic cells.  

\begin{definition}\label{defn: flow}
Let $Z$ be a regular switching parameter
\begin{enumerate}
\item
 The \emph{labeling map} $\cL:\LCC\times V \times \{-,+\}\to \{-1,1\}$ 
 describes the sign of the right hand side of the switching system
 on the cells that are neighbors of $\tau\in \LCC$ in a particular direction.  Letting $\rho = \rho^\tau$, 
 we first consider regular directions $j\notin \sd(\tau)$.
 %\Tomas{ This has still ramp functions in it and should be changed} 
 Here we look at the sign of the $j$-th equation of the switching system \eqref{eq: switch sys} on the boundary in the $j$-th direction
 \begin{align*}
    \cL(\tau,j,\beta) := \begin{cases}
        \sgn(-\gamma_j\theta_{a_j^\tau j} + \Lambda_j(\tau)), \quad & j\notin \sd(\tau),\; \beta = - \\
        \sgn(-\gamma_j\theta_{b_j^\tau j} + \Lambda_j(\tau)), \quad & j\notin \sd(\tau),\; \beta = + 
    \end{cases}
 \end{align*}
 For singular direction $j \in \sd(\tau)$, we look at a $j$-neighbor of $\tau$ and ask for the sign of the $\rho(j)$-th equation of the switching system because $\Lambda_{\rho(j)}$ is guaranteed to be well defined on a $j$-neighbor (see Lemma \ref{lem: sd labels}):
 \begin{align*}
    \cL(\tau,j,\beta) := \begin{cases}
        \sgn(-\gamma_{\rho(j)}\theta_{\rho^2(j)\rho(j)} + \Lambda_{\rho(j)}(\tau_j^-)), \quad & j\in\sd(\tau),\: \beta= - \\
        \sgn(-\gamma_{\rho(j)}\theta_{\rho^2(j)\rho(j)} + \Lambda_{\rho(j)}(\tau_j^+)), \quad & j\in\sd(\tau),\: \beta= +.
         \end{cases}
 \end{align*}

\item  The 
 \emph{flow direction map}, $\Phi: \LCC \to \{-1,0,1\}^N$ summarizes the degree of agreement in the labeling map  between the neighbors of $\tau$ in a given direction. It  is defined component-wise by
 \begin{align*}
  \Phi_j(\tau) := \begin{cases}
                     1, \quad &\cL(\tau,j,-) = 1 = \cL(\tau,j,+)\\
                     -1, \quad &\cL(\tau,j,-) = -1 = \cL(\tau,j,+)\\
                    0, \quad &\cL(\tau,j,-) =-\cL(\tau,j,+).
                   \end{cases}
 \end{align*}
   \end{enumerate}
\end{definition}

\red{
\begin{example}\label{ex: labeling map}
 Consider the positive toggle plus system of Example \ref{ex: positive toggle}. Let $\tau' = (0,\theta_{21}) \times (0,\theta_{12})$ be the lower left regular cell of the cell complex as in Figure \ref{fig: neighbors}(b). Then $\Lambda_2(\tau') = L_{22}L_{21}$ and $\theta_{b_2 2} = \theta_{12}$ so that  
\[
    \cL(\tau',2,+) = \sgn(-\gamma_2 \theta_{a_2 2}+\Lambda_2(\tau')) = \sgn(-\gamma_2 \theta_{12} + L_{22}L_{21}) = -1
\]
which is represented by the dotted down arrow originating from the upper boundary of $\tau'$ Figure \ref{fig: neighbors}(b). This indicates the flow of the switching system in the $x_2$ direction is downward when $x\in \tau'$ is close to the upper boundary of $\tau'$. We also have $\theta_{a_2 2} = 0$ so that
\[
    \cL(\tau',2,-) = \sgn(-\gamma_2 \theta_{a_2 2} + \Lambda_2(\tau')) = \sgn(0 + L_{22}L_{21}) = 1
\]
and $\Phi_2(\tau') = 0$ since $\cL(\tau',2,-) = -\cL(\tau',2,+)$. This indicates that $\dot{x}_2 = 0$ for some $x\in\tau'$. The cell $\tau = \{\theta_{21}\} \times \{\theta_{12}\}$ is a loop characteristic cell with $\rho(1) = 2$ and $\rho(2) = 1$. We compute
\[
    \cL(\tau,1,-) = \sgn(-\gamma_{\rho(1)}\theta_{\rho^2(1)\rho(1)} + \Lambda_{\rho(1)}(\tau_1^-)) = \sgn(-\gamma_2 \theta_{12} + L_{22}L_{21}) = -1. 
\]
which is the direction of the flow on the left neighbor of $\tau$. This is also represented by the dotted down arrow in Figure \ref{fig: neighbors}(b).  
\end{example}
}

%\red{\sout{The value of the labeling map $\cL(\tau,j,\pm )$, as well as flow direction map $\Phi$, can be represented on a diagram of the cell complex, see Figure \ref{fig: neighbors}(b).}} 
\red{In general,} $\cL(\tau,j,\pm)$ is represented \red{on a diagram of the cell complex}  by an arrow originating from $\tau_j^\pm$ pointing in direction $\rho(j)$ either positively or negatively according to the sign of $\cL(\tau,j,\pm)$. As suggested by Figure~\ref{fig: neighbors}, the flow direction map gives rise to a \emph{state transition graph} which represents admissible transitions between the states that are represented by $\kappa\in \chi^{(N)}$. The state transition graph is explicitly constructed in \cite{cummins16}.

Note that the flow direction map depends on the choice of parameter $Z$. \red{ A key observation is that  it  only depends on inequalities between parameters.}  We define an equivalence relation on all parameters $Z$ that \red{satisfy the same inequalities and therefore} produce the same flow direction map. These equivalence classes, which we now proceed to define, will be called \emph{combinatorial parameters}.

\begin{definition}[\red{Definitions 4.5 and 4.6 of} \cite{cummins16}]\label{defn: comb par}
Consider a regular switching parameter $Z$. 
\begin{enumerate}
    \item The \emph{input combinations} of the $i$th node is the Cartesian product \[\In_i := \prod_{j\in\Sources(i)}\{\off,\on\}.\]
   
     The \emph{indicator function}, $\one_i :\bR_+^{\Sources(i)}\to \In_i$, is defined component-wise by
    \begin{align*}
        \one_{ij}(x) := \begin{cases}
        \off, \quad & \bs_{ij} = 1 \mbox{ and } x_j<\theta_{ij} \mbox{ or } \bs_{ij}=-1 \mbox{ and } x_j>\theta_{ij} \\
                     \on, \quad & \bs_{ij}=1 \mbox{ and } x_j>\theta_{ij} \mbox{ or } \bs_{ij}=-1 \mbox{ and } x_j<\theta_{ij} \\
                     \mbox{undefined}, \quad &\mbox{otherwise}.
    \end{cases}
    \end{align*}
    
    The \emph{$\sigma$-valuation function}, $v_{i\red{j}}:\In_i\to  \red{\bR}$ %\red{\sout{$\bR^{\Sources(i)}$}}, 
    is defined by
    \begin{align*}
        v_{ij}(A) := \begin{cases}
            L_{ij}, \quad & \red{A} = \off \\
            U_{ij}, \quad & \red{A} = \on \\
            \mbox{undefined}, \quad & \mbox{otherwise}.
        \end{cases}
    \end{align*}
    Note that $\sigma_{ij} = v_{ij}\circ \one_{ij}$. The \emph{$\Lambda$-valuation function}, $\omega_i:\In_i \to \bR$, is defined by
    \begin{align*}
        \omega_i(A) := \prod_{\ell=1}^{p_i}\sum_{j\in I_\ell} v_{ij}(A\red{_j}). 
    \end{align*}
    Note that $\Lambda_i = \omega_i \circ \one_i$.
    
    Define $\sL_j:\In_j\times \Targets(j) \to \{-1,1\}$ by 
    \begin{align*}
        \sL_j(A,i) := \sgn(-\gamma_j\theta_{ij} + \omega_j(A)).
    \end{align*}
    The \emph{logic parameter} is the collection $\sL(Z) := (\sL_1(\cdot,\cdot),\ldots,\sL_N(\cdot,\cdot)\red{)}$. 
    
    \item We define an equivalence relation $Z \sim Z'$ whenever $(\sL(Z'),O(Z')) = (\sL(Z),O(Z)\red{)}$. The \emph{combinatorial parameter} is an equivalence class of this relationship. In other words, 
    $ Z' \in \cP(Z)$ whenever $(\sL(Z'),O(Z')) = (\sL(Z),O(Z))$.
\end{enumerate}
\end{definition}

  The notion of combinatorial parameter $\cP(Z)$ was introduced in \cite{cummins16}. Each combinatorial parameter  is defined in terms of inequalities between real valued parameters  of $Z$. Therefore each combinatorial parameter corresponds to an open domain in the real-valued parameter space of parameters $Z$.
  The key observation from  \cite{cummins16} is that any two parameters $Z_1, Z_2 \in \cP(Z)$ define identical labeling maps and therefore identical flow direction maps. This is because the logic parameters  $L_j$ that enter the definition of combinatorial parameters represent the same signs of the switching differential equations that define the labeling map.  
  
  \red{
  \begin{example}
    Consider the positive toggle plus system of Example \ref{ex: positive toggle}.  We will reference $\tau'$ and $\tau''$ from Figure \ref{fig: neighbors}(b). The input combinations for the first node is
    \[
        \In_1 = \{\off, \on\} \times \{\off, \on\}
    \]
    because the first node has two inputs. The indicator function depends only on the order parameter $O(Z)$. The second component of the indicator function $\one_1$ satisfies
    \[
        \one_{12}(x) = \begin{cases}
            \off, \quad x\in \tau' \\
            \on, \quad x\in \tau''
        \end{cases}.
    \]
    The $\sigma$-valuation function $v_{12}$ satisfies $v_{12}(\off) = L_{12}$ and $v_{12}(\on) = U_{12}$ so that 
    \[
        \sigma_{12}(x) = v_{12}(\one_{12}(x)) = \begin{cases}
            L_{12}, \quad x \in \tau'\\
            U_{12}, \quad x \in \tau''
        \end{cases}. 
    \]
    The $\Lambda$-valuation function $\omega_i$ satisfies 
    \begin{align*}
        \omega_1((\off,\off)) &= v_{11}(\off)v_{12}(\off) = L_{11}L_{12}, \mbox{ and}\\
        \omega_1((\off,\on)) &= v_{11}(\off)v_{12}(\on) = L_{11}U_{12}
    \end{align*}
    so that 
    \begin{align*}
        \Lambda_1(x) = \omega_1(\one_1(x)) = \begin{cases}
            \omega_1((\off,\off)) = L_{11}L_{12}, \quad x\in \tau' \\
            \omega_1((\off,\on)) = L_{11}U_{12}, \quad x \in \tau''
        \end{cases}.
    \end{align*}
    The first component of the logic parameter $\sL(Z)$ satisfies
    \begin{align*}
        &\sL_1((\off,\off),2) = \sgn(-\gamma_1\theta_{21} + \omega_1((\off,\off))) = \sgn(-\gamma_1\theta_{21} + L_{11}L_{12}) = -1, \mbox{ and}\\
        &\sL_1((\off,\on),2) = \sgn(-\gamma_1\theta_{21} + \omega_1((\off,\on))) = \sgn(-\gamma_1\theta_{21} + L_{11}U_{12}) = 1. 
    \end{align*}
    These values are related to the labeling map via $\cL(\tau',1,+) = \sL_1((\off,\off),2)$ and $\cL(\tau'',1,+) = \sL_1((\off,\on),2)$. 
  \end{example}
  }

%%%%%%%%%%%%%%%%%%%%%%%%%%%%%

\subsection{Characterization of Equilibrium Cells}

We now provide a theorem which characterizes the equilibrium cells of \red{the} switching system \red{$\SWITCH(Z)$} and shows that there is a unique equilibrium \red{of $\cS(Z,\eps)$ which converges to} each equilibrium cell. The proof for the theorem can be found in Section~\ref{sec: eq cell proofs}.

\begin{theorem}\label{thm: FP}
Let $Z$ be a regular switching parameter. 
\begin{description}
 \item[(a)] 
$\tau\in \chi$ is an equilibrium cell if and only if
 \begin{enumerate}
 \item $\tau$ is a loop characteristic cell, and
  \item $\Phi_j(\tau) = 0$ for each $j$. 
 \end{enumerate}
 \item[(b)]  Furthermore, \red{there is an $0<A\in \bR^{N\times N}$ so that for $\eps < A$} any sigmoidal system $\cS(Z,\eps)$ \red{has} a unique equilibrium $x^\eps$such that $x^\eps \to \tau$ as $\eps \to 0$. 
\end{description}
\end{theorem}

The theorem in the case of regular cells is implied by Proposition 3.6 of \cite{cummins16}. It was shown in \cite{veflingstad07} that loop characteristic cells are a subset of equilibrium cells in the case that the sigmoidal perturbations $\sigma_{ij}(\cdot,\eps)$ are Hill functions. This theorem extends these results by providing a necessary and sufficient condition for the identification of all equilibrium cells, enlarging the class of functions for which it applies, and giving uniqueness of the equilibrium $x^\eps$.

By Theorem \ref{thm: FP}, an equilibrium cell $\kappa$ has a unique equilibrium $x^\eps$ of $\cS(Z,\eps)$ associated to it. We associate the stability of this equilibrium with the cell through the following definition.

\begin{definition}\label{defn: cell stability}
    An equilibrium cell $\kappa$ is \emph{stable} if the associate equilibrium $x^\eps$ of $\cS(Z,\eps)$ is stable for all $\eps>0$ small enough and \emph{unstable} otherwise. 
\end{definition}

The equilibrium cells of a switching system can be computed  using the DSGRN software \cite{DSGRNgithub}.  We show in Section~\ref{sec: gen sys results} that the analysis of their stability can be reduced to the problem of stability of multiple cyclic feedback systems that are associated  to each singular equilibrium cell. We therefore first discuss the stability of equilibrium cells in cyclic feedback systems.

\section{Equilibrium Cells and their Stability in Cyclic Feedback Networks}\label{sec: cfs results}

This section 
concentrates on a particular type of a network, a cyclic feedback network, and characterizes the equilibrium cells and the stability of the equilibria they contain.  In the following section, we generalize these results to arbitrary networks. 

\begin{definition}\label{defn: cfn}
A \emph{cyclic feedback network} (CFN) is a regulatory network $\bRN=(V,E)$ such that $E = \{(1,2),(2,3),\ldots,(N-1,N),(N,1)\}$. A \emph{cyclic feedback system} (CFS) is a switching or sigmoidal system associated to a CFN.
 \end{definition}
 
Throughout this section we will assume that $\bRN$ is a cyclic feedback network. Since each node $j$ has exactly one target, $j+1$, and one source, $j-1$, the node $j$ is associated to exactly one threshold, $\theta_{(j+1)j}$, and $\Lambda_j = \sigma_{j(j-1)}$.  This implies the combinatorial parameter, and thus the flow direction map, is determined by the ordering of numbers within the sets $\{\gamma_j\theta_{(j+1)j},L_{j(j-1)},U_{j(j-1)}\}$, $j=1, \ldots, N \mod N$. This observation informs the following definition.

\begin{definition}
 Given a regular switching parameter $Z = (L,U,\theta,\Gamma)$  for a cyclic feedback system, a node $j$ is \emph{essential} if $L_{j(j-1)}<\gamma_j\theta_{(j+1)j}<U_{j(j-1)}$ and \emph{inessential} otherwise. 
\end{definition}

Another consequence of having exactly one threshold for each node is that there is only one singular loop characteristic cell, $\tau$, for which all the directions are singular, i.e. $\sd(\tau) = V$.  The permutation $\rho$ for this cell is defined by $\rho^\tau(j) = j+1$. 
Throughout this section, $\tau$ always denotes this cell and $\rho=\rho^\tau$ will denote the associated permutation.  We associate a sign to $\rho$ which describes whether the net effect of the cycle is positive or negative: 
\[ \sgn(\rho) := \prod_{i=1}^N \bs_{(i+1)i}.\]
We say $\bRN$ is a \emph{positive CFN} if $\sgn(\rho) = 1$ and a \emph{negative CFN} if $\sgn(\rho) = -1$.

\subsection{Equilibrium Cells}\label{sec: cfs equilibria}

 This section identifies equilibrium cells of a CFN.  To simplify notation, we observe as in \cite{gedeon94} that by changing variables we may assume without loss of generality that if $\bRN$ is a positive CFN, then every edge is activating, and if $\bRN$ is a negative CFN, then every edge is activating except the edge $(N,1)$. This change of variables is of the form 
\begin{align*}
    x_j\to \alpha_j(x_j-\theta_{(j+1)j})+\theta_{(j+1)j}, 
\end{align*}
where $\alpha_j = \pm 1$. 

The following lemma specifies the equilibrium cell for a CFS  at a parameter $Z$ for which it has an inessential node. 
\begin{lemma}\label{lem: inessential cfs}
 If at a switching parameter $Z=(L,U,\theta, \Gamma)$ the CFS has at least one inessential node, then $\SWITCH(Z)$ has a unique equilibrium cell $\kappa$ and this cell  is regular. 
 The cell $\kappa$ is defined as follows.
 If $j$ is an inessential node, the $j$-th projection is 
 \begin{align*}
     \pi_j(\kappa) = \begin{cases}
       (0,\theta_{(j+1)j}), \quad & \mbox{ if} \; U_{j(j-1)}<\gamma_j\theta_{(j+1)j}\\
       (\theta_{(j+1)j},\infty), \quad & \mbox{ if} \; \gamma_j\theta_{(j+1)j}<L_{j(j-1)}. 
     \end{cases}
 \end{align*}
 If $j$ is essential, let $k$ be the inessential node which forms the shortest path of the form $k\to k+1 \to\cdots \to j$, where nodes $k+1, \ldots, j-1$ are essential. We have two cases:
 
 If $\sgn(\rho) = 1$, then 
 \begin{align*}
     \pi_j(\kappa) = \begin{cases}
      (0,\theta_{(j+1)j}), \quad & \mbox{ if} \; U_{k(k-1)}<\gamma_k\theta_{(k+1)k},\\
       (\theta_{(j+1)j},\infty), \quad & \mbox{ if} \; \gamma_k\theta_{(k+1)k}<L_{k(k-1)}. 
     \end{cases}
 \end{align*}
 
 If $\sgn(\rho) = -1$, then  
 \begin{align*}
     \pi_j(\kappa) = \begin{cases}
      (0,\theta_{(j+1)j}), \quad & \mbox{ if} \; U_{k(k-1)}<\gamma_k\theta_{(k+1)k} \mbox{ and } 1\leq k< j, \mbox{ or } \gamma_k\theta_{(k+1)k}<L_{k(k-1)} \mbox{ and } j<k\leq N, \\
       (\theta_{(j+1)j},\infty), \quad & \mbox{ if} \; \gamma_k\theta_{(k+1)k}<L_{k(k-1)} \mbox{ and } 1\leq k<j, \mbox{ or } U_{k(k-1)}<\gamma_k\theta_{(k+1)k} \mbox{ and } j<k\leq N. 
     \end{cases}
 \end{align*}
\end{lemma}

\begin{figure}[htbp!]

\begin{tabular}{cc}

\begin{tikzpicture}[scale = 2.5]
  \tikzstyle{line} = [-,thick]
  \tikzstyle{arrow} = [->,line width = .4mm]
  \tikzstyle{unstable} = [red]
  \tikzstyle{stable} = [blue]
  \tikzstyle{pt} = [circle,draw=black,fill = black,minimum size = 2pt];
  \def\arrlen{.3}
  
  %vertices of cells
  \node[pt] at (0,0) (00) [label =below: 0,label = west: 0] {};
  \node[pt] at (1,0) (10) [label = below:$\theta_{21}$ ] {};
  \node[pt] at (0,1) (01) [label = west:$\theta_{12}$] {};
  \node[pt] at (2,0) (20) [label = below:$\infty$] {};
  \node[pt] at (1,1) (11) [label = 45:{\Large $\tau$}] {};
  \node[pt] at (2,1) (21) {};
  \node[pt] at (0,2) (02) [label = west:$\infty$] {};
  \node[pt] at (1,2) (12) {};
  \node[pt] at (2,2) (22) {};

  %draw faces of cells
  \draw[line] (00) to (20);
  \draw[line] (01) to (21);
  \draw[line] (02) to (22);
  
  \draw[line] (00) to (02);
  \draw[line] (10) to (12);
  \draw[line] (20) to (22);
  
 %wall labelings
  \draw[arrow] (1,.5) to (1-\arrlen,.5);
  \draw[arrow] (.5,1) to (.5,1-\arrlen);
  \draw[arrow] (1.5,1) to (1.5,1+\arrlen);
  \draw[arrow] (1,1.5) to (1+\arrlen,1.5);
  
  \node at (.5,.5) {\Large $\kappa^L$};
  \node at (1.5,1.5) {\Large $\kappa^H$};
  
 \end{tikzpicture}
 &

     \begin{tikzpicture}[scale = 2.5]
  \tikzstyle{line} = [-,thick]
  \tikzstyle{arrow} = [->,line width = .4mm]
  \tikzstyle{unstable} = [red]
  \tikzstyle{stable} = [blue]
  \tikzstyle{pt} = [circle,draw=black,fill = black,minimum size = 2pt];
  \def\arrlen{.3}
  
  %vertices of cells
  \node[pt] at (0,0) (00) [label =below: 0,label = west: 0] {};
  \node[pt] at (1,0) (10) [label = below:$\theta_{21}$ ] {};
  \node[pt] at (0,1) (01) [label = west:$\theta_{12}$] {};
  \node[pt] at (2,0) (20) [label = below:$\infty$] {};
  \node[pt] at (1,1) (11) [label = 45:{\Large $\tau$}] {};
  \node[pt] at (2,1) (21) {};
  \node[pt] at (0,2) (02) [label = west:$\infty$] {};
  \node[pt] at (1,2) (12) {};
  \node[pt] at (2,2) (22) {};

  %draw faces of cells
  \draw[line] (00) to (20);
  \draw[line] (01) to (21);
  \draw[line] (02) to (22);
  
  \draw[line] (00) to (02);
  \draw[line] (10) to (12);
  \draw[line] (20) to (22);
  
 %wall labelings
  \draw[arrow] (1,.5) to (1+\arrlen,.5);
  \draw[arrow] (.5,1) to (.5,1-\arrlen);
  \draw[arrow] (1.5,1) to (1.5,1+\arrlen);
  \draw[arrow] (1,1.5) to (1-\arrlen,1.5);
  
  \node at (.5,.5) {\Large $\kappa$};
  
 \end{tikzpicture}
 \\
 (a): Positive CFS & (b): Negative CFS
\end{tabular}

\caption{\textbf{Labeling map 
for the two node CFSs with no inessential nodes.} The labeling map on the boundary of $\bR_+^2$ points inward (not shown). \textbf{(a):} $\bRN = (V,E) = (\{1,2\},\; \{(1\to 2), (2\to 1)\}$.  The arrows indicate $\Phi_i(\kappa^L) = \Phi_i(\kappa^H) = \Phi_i(\tau) = 0$, $i=1,2$ so that $\kappa^L$, $\kappa^H$, and $\tau$ are equilibrium cells. \textbf{(b):} $\bRN = (V,E) = (\{1,2\}, \{(1\to 2), (2\dashv 1)\}$. The arrows indicate that $\Phi_2(\kappa) = 0$ but $\Phi_1(\kappa) = 1$ so that $\kappa$ is not an equilibrium cell. However, $\Phi_1(\tau)=\Phi_2(\tau) = 0$ so $\tau$ is an equilibrium cell. There are no regular equilibrium cells.  }\label{fig: essential CFS}
\end{figure}

When $N=2$, and  for a parameter $Z$ where all nodes are essential, the value of the labeling map on the neighbors of $\tau$  for a positive CFS  is depicted in Figure \ref{fig: essential CFS}(a), and for a negative CFS in Figure \ref{fig: essential CFS} (b). It is apparent that the positive CFS has two regular equilibrium cells and the negative CFS has no regular equilibrium cells. In either case $\tau$ is a singular equilibrium cell. The next lemma shows that this is true for all $N$. 

\begin{lemma}\label{lem: essential cfs}
     If at a switching parameter $Z=(L,U,\theta, \Gamma)$ the  CFS  has no inessential nodes, then $\tau$ is an equilibrium cell of $\SWITCH(Z)$. Furthermore,
     \begin{enumerate}
         \item 
 If $\bRN$ is a positive CFN then $\SWITCH(Z)$ has exactly two regular equilibrium cells defined by 
\begin{align*}
   \kappa^L = \prod_{j=1}^N (0,\theta_{(j+1)j}) \qquad  \mbox{ and } \qquad \kappa^H = \prod_{j=1}^N  (\theta_{(j+1)j},\infty). 
\end{align*}
\item
 If $\bRN$ is a negative CFN then $\tau$ is the unique equilibrium cell.
    \end{enumerate}
\end{lemma}

The proofs for these lemmas can be found in Section \ref{sec: cfn eq cell proofs}.

%%%%%%%%%%%%%%%%%%%%%%%%%%%%%%%

\subsection{Stability of Equilibria in Sigmoidal CFSs}\label{sec: cfs stability}

To determine stability of equilibria of $\cS(Z,\eps)$, we compute the characteristic polynomial of the Jacobian $J(\eps)$.  The structure of a cyclic feedback system imposes structure on $J$. In particular, we have
\begin{align}\label{eq: CFS Jacobian}
    J = \left(\begin{array}{cccc}
     -\gamma_{1} & & & \sigma_{1N}' \\
     \sigma_{21}' & -\gamma_2 \\
     & \ddots & \ddots \\
     & & \sigma_{N(N-1)}' & -\gamma_N
    \end{array}\right).
\end{align}
Before computing the characteristic polynomial we first note that we can obtain stability of any regular equilibrium cell $\kappa$ from $J$. 

\begin{proposition}\label{prop: reg cell stable}
 If $\kappa$ is a regular equilibrium cell then it is stable. 
\end{proposition}

\begin{proof}
If $x^\eps$ is an equilibrium of $\cS(Z,\eps)$ which converges to $\kappa$, then $\sigma_{j(j-1)}'(x^\eps;\eps)$ converges to $0$. Therefore, for $\eps$ small enough $J$ is strictly diagonally dominant with negative diagonal entries and thus all eigenvalues have negative real part. 
\end{proof}

We now give the characteristic polynomial of $J$, which can be computed using the Leibniz Formula for the determinant. For the proof of the following lemma and remaining results of this subsection, see Section \ref{sec: cfn stability proofs}.

\begin{lemma}\label{lem: cfs det(J)}
Let $\cS(Z,\eps)$ be a cyclic feedback system. The characteristic polynomial of the Jacobian $J(x;\eps)$ is given by
    \begin{align*}
        \det(J(x;\eps)-\lambda I) &= (-1)^N\left(\prod_{i=1}^N (\gamma_i+\lambda) -\sgn(\rho)M(x,\eps) \right)
    \end{align*}
    where $M(x,\eps) : = \prod_{i=1}^N |\sigma_{i(i-1)}'(x;\eps)|$. 
 \end{lemma}

The sign of the CFS plays a significant role in stability of equilibria. First we address positive cyclic feedback systems. 

\begin{proposition}\label{prop: pos cycle stability}
 Let $\bRN$ be a positive CFN and $x$ be an equilibrium of $\cS(Z,\eps)$. Stability of $x$ can be determined as follows.
 \begin{enumerate}
\item  If $M(x,\eps) < \prod_j \gamma_j$ then $x$ is asymptotically stable. 
\item If $M(x,\eps) > \prod_j \gamma_j$ then $x$ is unstable. 
\item If $M(x,\eps) = \prod_j \gamma_j$ then $\cS(Z,\eps)$ has a steady state bifurcation at $x$. 
 \end{enumerate}
\end{proposition}

 If an equilibrium $x^\eps$ of $\cS(Z,\eps)$ converges to the singular loop characteristic cell $\tau$, then condition (2) of Proposition \ref{prop: pos cycle stability} is satisfied if $\eps$ is small enough. Therefore, we have the following. 

\begin{proposition}\label{prop: pos sing cell unstable}
 Let $\bRN$ be a positive CFN. If $\tau$ is a singular equilibrium cell, then it is unstable. 
\end{proposition}

 Now we discuss negative cycles. \red{When there are $N\leq 2$ nodes, we can compute eigenvalues at an equilibrium and show that $\tau$ is stable. }

\begin{proposition}\label{prop: neg N<3 stable}
    Let $\bRN$ be a negative CFN with $N \leq 2$. \red{If $\tau$ is a singular equilibrium cell, then it is stable.} 
\end{proposition}

To address negative cyclic feedback systems with $N>2$ we make an additional assumption that the $\gamma$'s are identical, that is $\Gamma = I$. This allows us to explicitly compute all the eigenvalues of $J$. 

\begin{lemma}\label{lem: eigenvalues}
 Let  $\bRN$  be a CFN. Consider switching parameter $Z$ with 
$\Gamma = I$. Let $\lambda_{k}(x,\eps)$ for $k=0,\ldots,N-1$ be the eigenvalues of the Jacobian $J(x;\eps)$ evaluated at $x$. Then 
\begin{align}\label{eq: eigenvalues}
 \lambda_k(x,\eps) = \begin{cases}
                -1 + (e^{2\pi i k}M(x,\eps))^{\tfrac{1}{N}}, & \quad \sgn(\rho) = 1\\
                -1 + (e^{\pi i + 2\pi i k}M(x,\eps))^{\tfrac{1}{N}}, & \quad \sgn(\rho) = -1.
               \end{cases}
\end{align}
An eigenvalue with largest real part is 
\begin{align}
 \lambda_0(x,\eps) = \begin{cases}
                -1 + M(x,\eps)^{\frac{1}{N}}, & \quad \sgn(\rho) = 1\\
                -1 + (e^{\pi i}M(x,\eps))^{\frac{1}{N}}, & \quad \sgn(\rho) = -1.
               \end{cases}
\end{align}
\end{lemma}

The next proposition, often referred to as the secant condition for negative feedback systems \cite{tyson78,thron91,sontag02,sontag06,forger11}, follows immediately from the computation of the eigenvalues. 

\begin{proposition}\label{prop: neg cfs stability}
 Let $\bRN$ be a negative CFN, $Z = (L,U,\theta,\Gamma)$ be a switching parameter with $\Gamma = I$, and $x$ be an equilibrium of $\cS(Z,\eps)$.  If $N>2$, then stability can be determined as follows. 
 \begin{enumerate}
  \item If $M(x,\eps) < \sec\left(\tfrac{\pi}{N}\right)^N$ then $x$ is asymptotically stable. 
 
  \item If $M(x,\eps) > \sec\left(\tfrac{\pi}{N}\right)^N$ then $x$ is unstable. 
  
  \item If $M(x,\eps) = \sec\left(\tfrac{\pi}{N}\right)^N$ then $\cS(Z,\eps)$ has a Hopf bifurcation at $x$. 
 \end{enumerate} 
\end{proposition}

If $N>2$ and an equilibrium $x^\eps$ converges to the singular loop characteristic cell $\tau$, the second condition of Proposition \ref{prop: neg cfs stability} holds when $\eps$ is small enough and we have the following. 

\begin{proposition}\label{prop: neg sing cell stability}
 Let $\bRN$ be a negative CFN \red{with $N>2$} and \red{and $Z = (L,U,\theta,\Gamma)$ be a switching parameter with $\Gamma = I$}. If $\tau$ is a singular equilibrium cell it is unstable.
 %\Tomas{ Why the change for $N\leq 2$?}
\end{proposition}
\red{
We remark that while Proposition \ref{prop: neg sing cell stability} is a statement about sigmoidal CFS, \cite{farcot09} proves a stronger statement for switching CFS. In \cite{farcot09} it was shown that a negative switching CFS with $N>2$ and no inessential nodes has a stable periodic orbit and that the singular loop characteristic cell $\tau$ is an unstable source. This was proven even for $\Gamma \neq I$. We suspect Proposition \ref{prop: neg sing cell stability} holds for $\Gamma \neq I$ as well but do not pursue it here. We also suspect that for small $\eps$ the sigmoidal CFS with $N>2$ has a stable periodic orbit. This is consistent with Proposition \ref{prop: neg cfs stability} which suggests the existence of a supercritical Hopf bifurcation as $\eps$ increases. 

}

%%%%%%%%%%%%%%%%%%

\section{Equilibria, Stability, and Bifurcations in General Networks}\label{sec: gen sys results}

To characterize the equilibrium cells, stability and bifurcations of a network $\bRN$ which is not a cyclic feedback network, we decompose $\SWITCH(Z)$ locally on a neighborhood of an arbitrary loop characteristic cell into cyclic feedback systems and then apply the results of Section \ref{sec: cfs results}. 
Before proceeding to describe the decomposition, we  define the neighborhood of a cell on which the local decomposition is valid.

\begin{definition}\label{defn: cell neighborhood}
 For $\tau\in\chi$, the \emph{cell neighborhood} of $\tau$, denoted $\cN(\tau)$ is defined by
 \begin{align*}
     \cN(\tau) := \{\kappa\in\chi\;|\; \tau\subset\ol{\kappa}\}, 
 \end{align*} 
 where $\ol{\kappa}$ is the closure of $\kappa$.
\end{definition}

\red{We will write $x\in \cN(\kappa)$ to denote $x\in \tau$ for some $\tau \in \cN(\kappa)$. Note that the cell neighborhood of a regular cell $\kappa$ consists only of $\kappa$. In Figure \ref{fig: neighbors}(a), the cell neighborhood of the singular cell $\kappa$ is given by $\cN(\kappa) = \{\kappa, \kappa_1^-, \kappa_1^+\}$. The cell neighborhood of the singular cell $\til{\tau}$ of the positive toggle plus system consists of all cells contained in the interior of the gray shaded region of Figure \ref{fig: cones}(a). }

\subsection{Local Decomposition of $\SWITCH(Z)$ into Cyclic Feedback Systems}\label{sec: decomposition}

The idea behind the decomposition is to examine the values of $\Lambda$ on the cell neighborhood of a loop characteristic cell $\tau$. Using the next lemma we will show  that for any regular direction $r$ of $\tau$, $\Lambda_r$ is constant on $\cN(\tau)$. This will be used to show that regular directions enter trivially into the decomposition. On the other hand, for any singular direction $s$ of $\tau$, the lemma will be used to show that $\Lambda_{\rho(s)}$ takes one of two possible values. Furthermore, the value of $\Lambda_{\rho(s)}$ can only change in the $s$ direction. This will ultimately lead to the decomposition into cyclic feedback systems along the cycles in the permutation $\rho^\tau$.  See Section \ref{sec: eq cell proofs} for the proof of the lemma.

\begin{lemma}\label{lem: top cells}
Let $\tau\in\chi$ and $(j,i)\in E$.  If \red{$j$ is a regular direction of $\tau$ or $j\in \sd(\tau)$ and } $i\neq \rho^\tau(j)$ then 
 \begin{itemize}
     \item $\sigma_{ij}(\tau)$ is well defined,  
     \item for all $\kappa\in\cN(\tau)$, we have $\sigma_{ij}(\kappa) = \sigma_{ij}(\tau)$ is independent of $\kappa$.
 \end{itemize}

Consequently  if $ i \not \in \{\rho^\tau(j)\;|\; j\in\sd(\tau)\}$, then
 \begin{itemize}
     \item $\Lambda_i(\tau)$ is well defined,
     \item  for all $\kappa\in\cN(\tau)$ we have $\Lambda_i(\kappa) = \Lambda_i(\tau)$ is independent of $\kappa$, 
 \end{itemize}
\end{lemma}

 Let $\tau\in\chi$ be a loop characteristic cell, $\rho = \rho^\tau$ be the corresponding permutation.
 By Lemma \ref{lem: top cells}, 
 for each $x$ in a cell-neighborhood of $\tau$, $x\in\cN(\tau)$, and each regular direction $r$, the value of $\dot{x}_r$ in the switching system \eqref{eq: switch sys} is independent of every other variable. 
 
 Also by Lemma \ref{lem: top cells}, $\sigma_{sj}(\kappa) = \sigma_{sj}(\tau)$ is well defined for each singular direction $s$ and $ j \in \Sources(s)\setminus\{\rho^{-1}(s)\}$ on any cell $\kappa \in \cN(\tau)$. Therefore,  the value of $\dot{x}_s$ in the switching system \eqref{eq: switch sys} for any singular direction $s$ is independent of the regular directions.
 We conclude that on $\cN(\tau)$, the switching system $\SWITCH(Z)$ decomposes into two independent systems: one corresponding to the singular directions and one corresponding to the regular directions. Since for each regular direction $r$, $\dot{x}_r$ in  \eqref{eq: switch sys} depends only on $x_r$ itself, the regular directions system  consists of a collection of uncoupled one dimensional systems. 
 
 The dynamics of the singular directions decomposes according to the cycles that generate the permutation $\rho$. Let $\rho|_{\sd(\tau)} = (c_1,\ldots,c_n)$ be the cycle decomposition of $\rho$ restricted to the singular directions.  Let $\ell_d = \length(c_d)$ and $s_d = \sum_{j<d} \ell_j$.  We reorder the variables so that $c_d$ acts on $\{s_d+1, s_d+2,\ldots,s_d + \ell_d\}$ and $c_d(s_d+i) = s_d + i+1$ for $i<\ell_d$ and $c_d(s_d+\ell_d) = s_d + 1$. On $\cN(\tau)$, the dynamics of the variables $x^d = (x_{s_d+1},\ldots x_{s_d+\ell_d})$ are independent of the value of all other variables so that the dynamics of the singular directions can be written as $n$ independent systems. Within each system, $\dot{x}_{s}$ depends only on $\rho^{-1}(s)$ so each of these $n$ systems are cyclic feedback systems. 
 
 We conclude that the switching system $\SWITCH(Z)$ restricted to $\cN(\tau)$ decomposes into $d$ cyclic feedback systems and a diagonal system corresponding to regular directions. The sign of the cycle
 \[
 \sgn(c_d) := \prod_{j=s_d+1}^{s_d+\ell_d} \bs_{c_d(j)j}
 \]
 determines whether the $d$th system is a positive or negative CFS.

 To explicitly write the decomposition we define $\ell_{n+1} := N-s_{n+1}$ to be the number of regular directions. For $d=1,\ldots,n+1$ we define the projections of the cell neighborhood $\cN(\tau)$ and cells $\kappa \in \cN(\tau)$   
 \begin{align*}
 \cN^{d}(\tau) := \prod_{j=s_d + 1}^{s_d + \ell_d} \pi_j(\cN(\tau)) \qquad \mbox{and} \qquad \kappa^d := \prod_{j=s_d + 1}^{s_d+\ell_d} \pi_j(\kappa).
\end{align*}
  We set $\Lambda(\cdot;\tau):=\Lambda|_{\cN(\tau)}$ to be the restriction of $\Lambda$ onto $\cN(\tau)$. We then define 
  \[
   \Lambda^d(\cdot;\tau):=(\Lambda_{s_d+1}(\cdot;\tau),\ldots, \Lambda_{s_d + \ell_d}(\cdot;\tau))
  \]
  to be the projection of the resulting function  onto the directions in the $d$-th subsystem. 

  Further,  let $\Gamma^d$ be the $\ell_d\times \ell_d$ diagonal matrix with entries $\Gamma_{ii} = \gamma_{s_{d}+i}$ for $i=1,\ldots,\ell_d$. 
  The dynamics for the $d$-th subsystem is given explicitly by 
\begin{align}\label{eq: cycle sys}
    \dot{x}^d = -\Gamma^d x^d + \Lambda^d(x^d;\eps,\tau), \quad x^d\in \cN^d(\tau;\eps) 
\end{align}
which we denote by $\SWITCH^d(Z;\tau)$. On the collection of cells $\cN(\tau)$, the system is  
\[
\SWITCH(Z;\tau) := (\SWITCH^1(Z;\tau),\ldots,\SWITCH^n(Z;\tau),\SWITCH^{n+1}(Z;\tau)). 
\]
 We denote the cyclic feedback network associated with $\SWITCH^d(Z;\tau)$ by $\bRN^d(\tau)$. 

\red{
\begin{example}\label{ex: decomp}
    Consider the positive toggle plus system of Example \ref{ex: positive toggle}. For the loop characteristic cell $\til{\tau} = \{\theta_{11}\}\times \{\theta_{22}\}$ pictured in Figure \ref{fig: neighbors}(b), the singular directions of $\til{\tau}$ are both directions $1$ and $2$ and the cycle decomposition is given by $\rho = (c_1,c_2)$ where $c_1(1) =1$ and $c_2(2) = 2$. The cycle decomposition corresponds to the disjoint loops $1\to1$ and $2\to 2$ in the network. The cell neighborhood of $\til{\tau}$, $\cN(\til{\tau})$ is the gray shaded region in Figure \ref{fig: cones}(a). For $x\in \cN(\til{\tau})$, $x_1>\theta_{21}$ and $x_2 > \theta_{12}$ so that 
    \begin{align*}
        \dot{x}_1 &= -\gamma_1 x_1 + \Lambda_1(x) = -\gamma_1 x_1 + U_{12}\sigma_{11}(x_1) \\
        \dot{x}_2 &= -\gamma_1 x_2 + \Lambda_2(x) = -\gamma_2 x_2 + U_{21}\sigma_{22}(x_2).
    \end{align*}
    The decomposition at $\til{\tau}$, $\SWITCH(Z) = (\SWITCH^1(Z;\til{\tau}), \SWITCH^2(Z;\til{\tau}))$ is defined by $x^1 = x_1$, $x^2 = x_2$, $\Gamma^1 = \gamma_1$, $\Gamma^2 = \gamma_2$, $\Lambda^1(x_1;\til{\tau}) = U_{12}\sigma_{11}(x_1)$, and $\Lambda^2(x_2;\til{\tau}) = U_{21}\sigma_{22}(x_2)$. Note that the superscripts index the cycles, not the directions.  
\end{example}
}

Next we show every regular cell  in $\chi$  lies in a neighborhood of at least one singular loop characteristic cell.  
This implies that every regular cell will be in at least one neighborhood $\cN(\tau)$ on which decomposition into cyclic feedback systems is applicable. Therefore, in spite of of being local, the  decomposition into cyclic feedback systems affects all the cells in $\chi$.  In particular, all equilibrium cells of $\SWITCH(Z)$ are contained in at least one neighborhood $\cN(\tau)$.  

\begin{lemma}
Suppose that every node of $\bRN$ has  an out-edge.  Then for every regular cell $\kappa \in \chi^{(N)}$, there is a singular loop characteristic cell $\tau \in \chi$ so that $\kappa \in \cN(\tau)$. 
\end{lemma}

\begin{proof}
    Let $\kappa$ be a regular cell. 
    Recall that for every $j$, $\pi_j(\kappa) = (\theta_{a_j^\kappa j},\theta_{b_j^\kappa j})$ where $a_j^\kappa,b_j^\kappa \in V\cup \{-\infty, \infty\}$. Since every $j$ has an out-edge and hence a target node $i_j$, there is a  threshold $\theta_{i_j j}$ so that $\theta_{i_j j}$ is a boundary of $\pi_j(\kappa)$. In particular, since $i_j$ is a network node, $i_j\notin \{\theta_{-\infty j}, \theta_{\infty j}\}$. 
    The selection of such a threshold for every $j \in V$ defines a map 
   $\sigma: V \to V$ by $\sigma(j) = i_j$.  
    Note that since $\sigma (V) \subset V$ and $V$ is finite, the set 
   \[ U:= \bigcup_{k=1}^\infty \sigma^k(V) \]
   is non-empty and satisfies $U=\sigma(U)$.
   Then a cell $\tau$ defined  by 
    \begin{align*}
        \pi_j(\tau) = \begin{cases}
          \pi_j(\kappa), \quad & j\notin U \\
          \{\theta_{\sigma(j)j}\}, \quad & j\in U
        \end{cases}
    \end{align*}
    is a loop characteristic cell with $\kappa \in \cN(\tau)$. 
\end{proof}

\subsection{Equilibrium Cells}

First, we generalize Section \ref{sec: cfs equilibria} and characterize the equilibrium cells of $\SWITCH(Z)$ in $\cN(\tau)$. We will take advantage of the decomposition \[\SWITCH(Z) = (\SWITCH^1(Z;\tau),\ldots,\SWITCH^n(Z;\tau),\SWITCH^{n+1}(Z;\tau))\] which is valid on $\cN(\tau) = \cN^1(\tau)\times \cdots \times \cN^n(\tau)\times \cN^{n+1}(\tau)$ (equation \ref{eq: cycle sys}). However, the lemmas of Section \ref{sec: cfs equilibria} are only valid for cyclic feedback systems which are defined on the whole positive orthant. We therefore extend $\SWITCH^d(Z;\tau)$ to all of $\bR_+^{\ell_d}$ for $d\leq n$. The regular directions do not form a cyclic feedback system on $\cN(\tau)$ so we do not need to extend $\SWITCH^{n+1}(Z;\tau)$.  

The dynamics of $\SWITCH^d(Z;\tau)$ are given by 
\begin{align*}
    \dot{x}^d = -\Gamma^d x^d + \Lambda^d(x^d;\tau), \quad x^d \in \cN^d(\tau).
\end{align*}
To extend the domain of definition of $\SWITCH^d(Z;\tau)$ on $\bR_+^{\ell_d}$, we need to define $\Lambda^d(\cdot;\tau)$ on $\bR_+^{\ell_d}$, while ensuring $\SWITCH^d(Z;\tau)$ remains a cyclic feedback system. To do so, we make the following definition.

\begin{definition}
 Given a loop characteristic cell $\tau \subset \bR_+^N$ with permutation $\rho = \rho^\tau$, the \emph{cone} $\sC(\kappa;\tau)$ rooted in $\tau$ and induced by a cell $\kappa\in\cN(\tau)$ is defined by its $N$ projections. For a regular direction, $r$, of $\tau$ 
 \begin{align*}
     \pi_r(\sC(\kappa;\tau))  := \pi_r(\tau). 
 \end{align*}
 For a singular direction, $s\in\sd(\tau)$, 
 \begin{align*}
     \pi_s(\sC(\kappa;\tau)) := \begin{cases}
     \{\theta_{\rho(s)s}\}, \quad & \mbox{ if } \pi_s(\kappa) = \{\theta_{\rho(s)s}\}\\
        (\theta_{\rho(s)s},\infty), \quad & \mbox{ if } \pi_s(\kappa) = (\theta_{\rho(s)s},\theta_{\rho_+(s)s}) \\
        (0,\theta_{\rho(s)s}), \quad &  \mbox{ if } \pi_s(\kappa) = (\theta_{\rho_-(s)s},\theta_{\rho(s)s}). 
     \end{cases}
 \end{align*}
 For  $\sigma \in \cN^d(\tau)$ define a $d$-cone  in $\bR_+^{\ell_d}$ by 
 \begin{align*}
     \sC^d(\sigma;\tau) := \prod_{j=s_d + 1}^{s_d +\ell_d} \pi_j(\sC(\sigma;\tau)).
 \end{align*}
\end{definition}

\begin{figure}[hbtp!]
\begin{center}
\begin{tabular}{ccc}
\begin{tikzpicture}[scale = 1.5]
  \tikzstyle{line} = [-, thick]
  \tikzstyle{arrow} = [->,line width = .4mm]
  \tikzstyle{unstable} = [red]
  \tikzstyle{stable} = [blue]
  \tikzstyle{pt} = [circle,draw=black,fill = black,minimum size = 2pt];
  \tikzstyle{cone pt} = [circle, draw = black, minimum size = 1pt, fill = blue];
  \def\epsshift{.3}
  \def\arrlen{.3}
    
  %vertices of cells
  \node[pt] at (0,0) (00) [label =below: 0,label = west: 0] {};
  
  \node[pt] at (1,0) (10) [label = below:$\theta_{21}$ ] {};
  \node[pt] at (2,0) (20) [label = below:$\theta_{11}$ ] {};
  
  \node[pt] at (1,1) (11)  {};
  \node[pt] at (2,1) (21) {};
  \node[pt] at (2,2) (22) [label = 45: $\til{\tau}$] {};
  \node[pt] at (2,1) (21) {};
  \node[pt] at (1,2) (12) {};
  \node[pt] at (3,0) (30) [label = below: $\infty$] {};
  \node[pt] at (3,1) (31) {};
  \node[pt] at (3,2) (32) {};

  \node[pt] at (0,1) (01)  [label = west: $\theta_{12}$] {};
  \node[pt] at (0,2) (02)  [label = west: $\theta_{22}$] {};
  \node[pt] at (0,3) (03) [label = west: $\infty$] {};
  \node[pt] at (1,3) (13) {};
  \node[pt] at (2,3) (23) {};
  \node[pt] at (3,3) (33) {};

  %draw faces of cells
  \draw[line] (00) to (30);
  \draw[line] (00) to (03);
  \draw[line] (10) to (13);
  \draw[line] (20) to (23);
  \draw[line] (30) to (33);
  
  \draw[line] (01) to (31);
  \draw[line] (02) to (32);
  \draw[line] (03) to (33);
  
  %label \kappa\in \cN(\tau)
  \node at (2.5,2.5) {$\kappa_1$};
  \node at (1.5,2.5) {$\kappa_2$};
  \node at (1.5,1.5) {$\kappa_3$};
  \node at (2.5,1.5) {$\kappa_4$}; 
  
  %arrows
    \draw[arrow] (.5,1) to (.5,1-\arrlen);
    \draw[arrow] (1,.5) to (1-\arrlen,.5);
    \draw[arrow] (1.5,1) to (1.5,1+\arrlen);
    \draw[arrow] (1,1.5) to (1-\arrlen,1.5);
    \draw[arrow] (.5,2) to (.5,2-\arrlen);
    \draw[arrow] (2,.5) to (2-\arrlen,.5);
    \draw[arrow] (1.5,2) to (1.5,2-\arrlen);
    \draw[arrow] (2.5,2) to (2.5,2-\arrlen);
    \draw[arrow] (2.5,1) to (2.5,1+\arrlen);
    \draw[arrow] (1,2.5) to (1-\arrlen,2.5);
    \draw[arrow] (2,2.5) to (2-\arrlen,2.5);
    \draw[arrow] (2,2.5) to (2+\arrlen,2.5);
    \draw[arrow] (2,1.5) to (2-\arrlen,1.5);
    \draw[arrow] (2,1.5) to (2+\arrlen,1.5);
  
  \begin{pgfonlayer}{background}
    \draw[fill = gray!50!white] (1,1) rectangle (3,3);
  \end{pgfonlayer}  
\end{tikzpicture} 
 
 &
 
 \begin{tikzpicture}[scale = 1.5]
  \tikzstyle{line} = [-, thick]
  \tikzstyle{arrow} = [->,line width = .4mm]
  \tikzstyle{unstable} = [red]
  \tikzstyle{stable} = [blue]
  \tikzstyle{pt} = [circle,draw=black,fill = black,minimum size = 2pt];
  \tikzstyle{cone pt} = [circle, draw = black, minimum size = 1pt, fill = blue];
  \def\epsshift{.3}
  \def\arrlen{.3}
    
  %vertices of cells
  \node[pt] at (0,0) (00) [label =below: 0,label = west: 0] {};
  
  \node[pt] at (2,0) (20) [label = below:$\theta_{11}$ ] {};
  
  \node[pt] at (2,2) (22) [label = {45: $\til{\tau}$} ] {};
  \node[pt] at (3,0) (30) [label = below: $\infty$] {};
  \node[pt] at (3,2) (32) {};
  
  \node[pt] at (0,2) (02)  [label = west: $\theta_{22}$] {};
  \node[pt] at (0,3) (03) [label = west: $\infty$] {};
  \node[pt] at (2,3) (23) {};
  \node[pt] at (3,3) (33) {};
  
  %draw faces of cells
  \draw[line] (00) to (30);
  \draw[line] (00) to (03);
  \draw[line] (20) to (23);
  \draw[line] (30) to (33);
  
  \draw[line] (02) to (32);
  \draw[line] (03) to (33);
  
  %label \sC(\kappa;\tau)
 \node at (1,2.5) {$\sC(\kappa_2;\til{\tau})$};
 \node at (1,1) {$\sC(\kappa_3;\til{\tau})$};
 \node at (2.5,.7) {$\sC(\kappa_4;\til{\tau})$};
 \node at (2.5,2.7) {$\sC(\kappa_1;\til{\tau})$};
 
 %arrows
 \draw[arrow] (2,1) to (2-\arrlen,1);
 \draw[arrow] (2,1) to (2+\arrlen,1);
 \draw[arrow] (1,2) to (1,2-\arrlen);
 \draw[arrow] (2.5,2) to (2.5,2-\arrlen);
 \draw[arrow] (2,2.5) to (2-\arrlen,2.5);
 \draw[arrow] (2,2.5) to (2+\arrlen,2.5);

 \end{tikzpicture} &
 
 \begin{tikzpicture}[scale = 1.5]
  \tikzstyle{line} = [-, thick]
  \tikzstyle{arrow} = [->,line width = .6mm]
  \tikzstyle{unstable} = [red]
  \tikzstyle{stable} = [blue]
  \tikzstyle{pt} = [circle,draw=black,fill = black,minimum size = 2pt];
  \tikzstyle{cone pt} = [circle, draw = black, minimum size = 1pt, fill = blue];
  \def\epsshift{.3}
  \def\arrlen{.5}
  %vertices of cells

  \node[pt] at (2,0) (20) [label = left:$0$, label = below: \vphantom{$\theta_{11}$} ] {};
  \node[pt] at (2,2) (22) [label = {east: $\til{\tau}^2$},label = {west: $\theta_{22}$}] {};
  \node[pt] at (2,3) (23) [label = west: $\infty$] {};
  
  %draw faces of cells
  \draw[line] (20) to node[right] {$\pi_2(\til{\tau}_2^-)$} (22);
  \draw[line] (22) to node[right] {$\pi_2(\til{\tau}_2^+)$} (23);
  
  %arrows
  \draw[arrow] (2,2) to (2,2-\arrlen);
  
 \end{tikzpicture} \\
 
 (a) &
 (b) &
 (c) \\
 \\
 
 \multicolumn{3}{c}{
 \begin{tikzpicture}[scale = 1.5]
  \tikzstyle{line} = [-, thick]
  \tikzstyle{arrow} = [->,line width = .6mm]
  \tikzstyle{unstable} = [red]
  \tikzstyle{stable} = [blue]
  \tikzstyle{pt} = [circle,draw=black,fill = black,minimum size = 2pt];
  \tikzstyle{cone pt} = [circle, draw = black, minimum size = 1pt, fill = blue];
  \def\epsshift{.3}
  \def\arrlen{.5}
  %vertices of cells
 
  \node[pt] at (0,2) (02) [label = below:$0$] {};
  \node[pt] at (2,2) (22) [label = {below: $\theta_{11}$}] {};
  \node[pt] at (3,2) (23) [label = below: $\infty$] {};
  
  \node at (1.8,2.2) {$\til{\tau}^1$};
  
  %draw faces of cells
  \draw[line] (02) to node[above] {$\pi_1(\til{\tau}_1^-)$} (22);
  \draw[line] (22) to node[above] {$\pi_1(\til{\tau}_1^+)$} (23);
  
  %arrows
  \draw[arrow] (2,2) to (2-\arrlen,2);
  \draw[arrow] (2,2) to (2+\arrlen,2);
  
 \end{tikzpicture}
 }\\
 \multicolumn{3}{c}{(d)} 
 \end{tabular}
 
 \caption{\red{\textbf{The decomposition of the positive toggle plus system at a loop characteristic cell $\til{\tau}$.} The switching parameter is chosen as in Example \ref{ex: positive toggle}. \textbf{(a):} The cell complex $\chi$ and the labeling map. The cell neighborhood $\cN(\til{\tau})$ consists of all cells in the shaded gray region. The regular cells in the cell neighorhood are labeled. \textbf{(b):} The cell complex and labeling map formed by the product of switching systems $(\SWITCH^1(Z;\til{\tau}),\SWITCH^2(Z;\til{\tau}))$ after extending their domains. The cell complex is formed by the cones rooted in $\til{\tau}$. The cones induced by the regular cells are labeled. \textbf{(c):} The cell complex $\chi^2(\tau)$ and labeling map $\cL^2$ for $\SWITCH^2(Z;\til{\tau})$. The cell $\til{\tau}_2^-$ is a $2$-candidate equilibrium cell because $C^2(\til{\tau}_2^-;\til{\tau})$ is an equilibrium cell of $\SWITCH^2(Z;\til{\tau})$. It is also $2$-consistent.   \textbf{(d):} The cell complex $\chi^1(\tau)$ and labeling map $\cL^1$ for $\SWITCH^1(Z;\til{\tau})$. Each of $\til{\tau}$, $\til{\tau}_1^-$, and $\til{\tau}_1^+$ are candidate equilibrium cells. The cells $\til{\tau}$ and $\til{\tau}_1^+$ are $1$-consistent, while $\til{\tau}_1^-$ is not. See Example \ref{ex: candidate eq cells} for more details on \textbf{(a)} and \textbf{(b)}. }
 %\Tomas{ This is a great example. Can you label cells $\til{\tau}_2^-$ in (c) and $\til{\tau}$, $\til{\tau}_1^-, \til{\tau}_1^+$ in (d)? It may be more important than the cones if you cannot fit all of that in. } 
 }\label{fig: cones}
\end{center}
   
\end{figure}
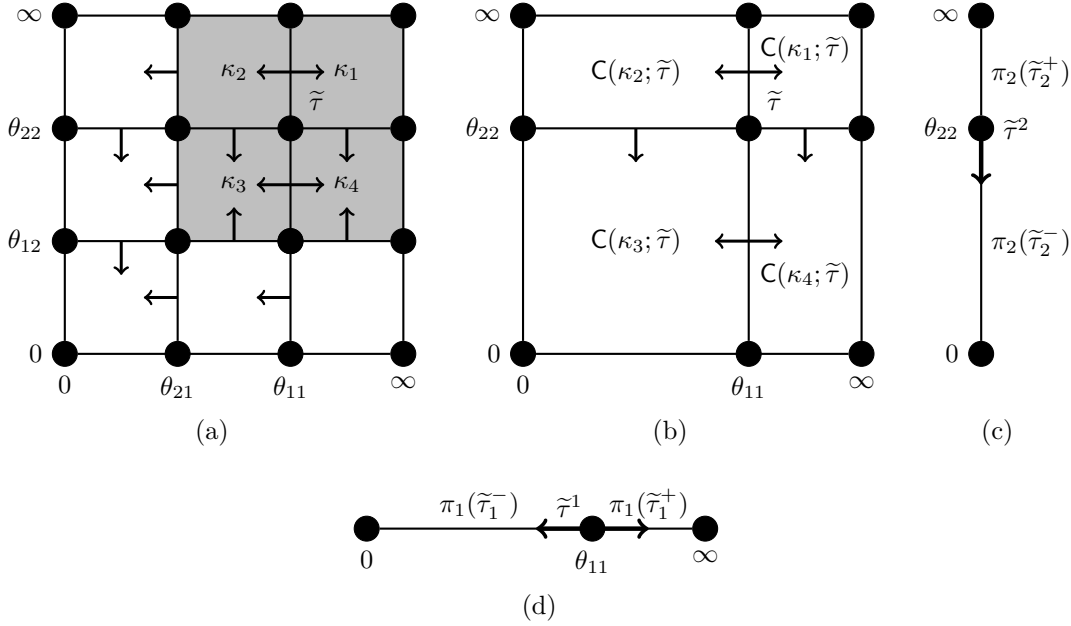

The cones for a loop characteristic cell $\tau$ in a two node network are depicted in Figure \ref{fig: cones}.

We now proceed to extend the function $\Lambda^d$ from $\cN^d(\tau)$ to $\bR_+^{\ell_d}$. Take $x^d \in \bR_+^{\ell_d}$.
If $x^d\notin \cN^d(\tau)$, then $x^d\in \sC^d(\sigma;\tau)$ for some $\sigma \in \cN^d(\tau)$. We define the value of $\Lambda^d$ on $\sC^d(\sigma;\tau)$ to be the value of $\Lambda^d$ on $\sigma$. Explicitly,  
\begin{align*}
    \Lambda^d(x^d;\tau) := \Lambda^d(\sigma), \quad x^d\in \sC^d(\sigma;\tau).   
\end{align*}

Given this extension, $\SWITCH^d(Z;\tau)$ is a switching system defined on $\bR_+^{\ell_d}$ with  an associated cell complex, labeling map, and flow direction map which we denote \red{$\chi^d(\tau)$}, $\cL^d(\cdot,\cdot,\cdot;\tau)$, and $\Phi^d(\cdot;\tau)$, respectively. Here, the threshold set \red{which generates $\chi^d(\tau)$,} $\Theta^d := (\Theta^d_{s_d + 1},\ldots,\Theta_{s_d+\ell_d})$, is defined by 
\begin{align*}
    \Theta_j^d := \begin{cases}
        \{0,\theta_{\rho(j)j},\infty\}, \quad & j\in\sd(\tau)\\
        \{\theta_{a_j^\tau j},\theta_{b_j^\tau \red{j}}\}, \quad & \mbox{otherwise}.
    \end{cases}
\end{align*}
Lemmas \ref{lem: inessential cfs} and \ref{lem: essential cfs} can be used to determine the equilibrium cells of $\SWITCH^d(Z;\tau)$. Since the cell complex is formed by the cones (see Figure \ref{fig: cones}\red{(c) and (d)}) 
\[
\chi^d(\tau) = \{\sC^d(\sigma;\tau)\;|\;  \sigma \in \cN^d(\tau)\},
\]
there is a straightforward identification between cells of $\SWITCH^d(Z;\tau)$ and cells $\sigma \in \cN^d(\tau)$.

\begin{definition}
 The \emph{$d$-candidate equilibrium cells} of $\tau$ are defined by 
 \begin{align*}
     \Eq^d(\tau) := \{\sigma^d\;|\; \sigma \in \cN(\tau) \mbox{ and } \sC^d(\sigma;\tau) \mbox{ is an equilibrium cell of } \SWITCH^d(Z;\tau)\}. 
 \end{align*}
 
 The \emph{candidate equilibrium cells} of $\tau$ are $\Eq(\tau) := \prod_{d=1}^{n+1} \Eq^d(\tau). $
\end{definition}

\red{See Figure \ref{fig: cones} and Example \ref{ex: candidate eq cells} for examples of candidate equilibrium cells in the positive toggle plus system.} The next lemma justifies the name candidate equilibrium cell. 

\begin{lemma}\label{lem: Eq}
 If $\kappa\in \cN(\tau)$ is an equilibrium cell, then $\kappa \in \Eq(\tau)$. 
\end{lemma}

Before proving the lemma, we isolate part of the proof that will be useful later. 

\begin{lemma}\label{lem: Phi and Phid equal}
    Let $\kappa\in\cN(\tau)$. If $r$ is a regular direction of $\tau$ then $\Phi_r(\kappa) = \Phi_r^{n+1}(C(\kappa;\tau);\tau)$.  For any singular direction $s$  of $\kappa$ that belongs to  $ \{s_d+1,\ldots,s_d+\ell_d\}$ we have $\Phi_s(\kappa) = \Phi_s^d(C(\kappa;\tau);\tau)$.  
\end{lemma}

\begin{proof}
    First consider the case that $r$ is a regular direction of $\tau$. Then $r$ is also a regular direction of $\kappa$ and $\pi_r(\kappa) = \pi_r(\tau) = (\theta_{a_r^\tau,r}, \theta_{b_r^\tau r})$. Further, by Lemma \ref{lem: top cells}, $\Lambda_r(\tau) = \Lambda_r(\kappa)$. Finally we have $\Lambda_r^{n+1}(C(\kappa;\tau);\tau) = \Lambda_r(\kappa) = \Lambda_r(\tau)$. These observations allow us to compute the labeling map of $\kappa$ in the $r$-th direction. We have
    \[
    \cL^{n+1}(C(\kappa;\tau),r,-) = \sgn(-\gamma_r \theta_{a_r^\tau r} + \Lambda_r^{n+1}(C(\kappa;\tau);\tau)) = \sgn(-\gamma_r\theta_{a_r^\tau r} + \Lambda_r(\tau)) = \cL(\tau,r,-)
    \]
    and 
    \[
    \cL^{n+1}(C(\kappa;\tau),r,+) = \sgn(-\gamma_r \theta_{b_r^\tau r} + \Lambda_r^{n+1}(C(\kappa;\tau);\tau)) = \sgn(-\gamma_r\theta_{b_r^\tau r} + \Lambda_r(\tau)) = \cL(\tau,r,+).
    \]
    Since the labeling maps agree and the flow direction maps are defined via the labeling maps (see Definition \ref{defn: flow}), we have $\Phi_r(\tau) = \Phi_r^{n+1}(C(\kappa;\tau);\tau)$.

    Now suppose $s \in  \{s_d+1,\ldots,s_d+\ell_d\}$ is a singular direction of $\kappa$ for some $d=1, \ldots, n$. Then 
\[
    \cL(\kappa,s,-) = \sgn(-\gamma_{\rho(s)}\theta_{\rho^2(s)\rho(s)} + \Lambda_{\rho(s)}(\kappa_s^-)) \quad \mbox{and} \quad  \cL(\kappa,s,+) = \sgn(-\gamma_{\rho(s)}\theta_{\rho^2(s)\rho(s)} + \Lambda_{\rho(s)}(\kappa_s^+)).
\]
Since $\Lambda_{\rho(s)}(\kappa_s^\pm) = \Lambda^d_{\rho(s)}(C(\kappa_s^{\pm};\tau);\tau)$, we have $\cL^d(C(\kappa;\tau),s,-) = \cL^d(C(\kappa;\tau),s,-)$ and $\cL(\kappa,s,+) = -\cL^d(C(\kappa;\tau),s,+)$, or $\Phi_s(\kappa) = \Phi_s^d(C(\kappa;\tau);\tau)$. 
\end{proof}

\begin{proof}[Proof of Lemma \ref{lem: Eq}]
Let $\kappa\in \cN(\tau)$ be an equilibrium cell. By Theorem \ref{thm: FP} we have $\Phi_j(\kappa) = 0$ for every $j\in V$. By Lemma \ref{lem: Phi and Phid equal}, we have
\[\Phi_r^{n+1}(C(\kappa;\tau);\tau) = \Phi_r(\kappa) =0\qquad \mbox{and} \qquad \Phi_s^d(C(\kappa;\tau);\tau) = \Phi_s(\kappa) = 0\]
 both for regular directions $r$ of $\tau$ and singular directions $s$ of $\kappa$ for some $d=1, \ldots, n$.  This implies $\kappa^{n+1}\in\Eq^{n+1}(\tau)$ and $\kappa^d\in\Eq^d(\tau)$ where $s\in\{s_d+1,\ldots,s_d+\ell_d\}$. 

Suppose $j \in \{s_d+1,\ldots,s_d+\ell_d\}$ is a singular direction of $\tau$ but a regular direction of $\kappa$. Assume without loss of generality that $\pi_{\rho(j)}(\kappa) = (\theta_{\rho_-(j)j},\theta_{\rho(j)j})$. Then we have 
\[
1 = \cL(\kappa,j,-) = \sgn(-\gamma_{j}\theta_{\rho_-(j)j} + \Lambda_{j}(\kappa)) = -\sgn(-\gamma_{j}\theta_{\rho(j)j} + \Lambda_{j}(\kappa)) = -\cL(\kappa,j,+). 
\]
Now, $\cL^d(C(\kappa;\tau),j,+) = \cL(\kappa,j,+)$  so we only need to check $\cL(\kappa,j,-) = 1 = \cL^d(\kappa^d,j,-)$. Since $\Theta_j^d = \{0,\theta_{\rho(j)j},\infty\}$, we have \[\cL^d(C(\kappa;\tau),\rho(j),-) = \sgn(-\gamma_{\rho(j)}\cdot 0 + \Lambda_{\rho(j)}(\kappa^d)) = 1\]  so that $\Phi^d(\kappa^d;\tau) = 0$ and $\kappa^d\in \Eq^d(\tau)$. Since $\kappa^d\in \Eq^d(\tau)$ for each $d$, $\kappa \in \Eq(\tau)$.  
\end{proof}

Once we have identified a candidate equilibrium cell $\kappa\in\Eq(\tau)$, there is one more property to check to ascertain that it is an actual equilibrium cell. Since $\kappa^d\in \Eq^d(\tau)$, the cell  $\sC^d(\kappa^d;\tau)$ must be  an equilibrium cell of $\SWITCH^d(Z;\tau)$. However, $\sC^d(\kappa^d;\tau)$ is a super set of $\kappa^d$, so this equilibrium may not be contained in $\kappa^d$. 

\begin{definition}
 Let $\sigma \in \Eq^d(\tau)$ and $d\leq n$. If the equilibrium of $\SWITCH^d(Z;\tau)$ in $\sC^d(\sigma;\tau)$ is contained in $\sigma$, then we say $\sigma$ is a \emph{$d$-consistent candidate equilibrium cell}. 
 \end{definition}

The next proposition shows how to check if $\sigma$ is a $d$-consistent candidate equilibrium cell. Notice that the condition of the proposition is vacuously satisfied when $\sigma$ is a singular equilibrium cell of $\SWITCH^d(Z;\tau)$. In other words, if the singular loop characteristic cell $\tau^d$ is an equilibrium cell, it is always $d$-consistent.  It is also important to note, that while seemingly  complicated, this condition is readily algorithmically checkable.

\begin{proposition}\label{prop: d-consistent}
 Let $d\leq n$ and $\sigma \in \Eq^d(\tau)$. Then $\sigma$ is a $d$-consistent candidate equilibrium cell if, and only if, for each $j\in\{s_d+1,\ldots,s_d+\ell_d\}$, 
if $\pi_j(\sigma)$ is equal to the $i$-th value in first column, and $\pi_{j-1}(\sigma)$ is equal to the $k$-th value in first row, then the inequality in entry $(i,k)$ of the following table is satisfied.

 \begin{center}
 \begin{tabular}{c| @{\hspace{.2in}} c @{\hspace{.2in}} c} 
  \diagbox{$\pi_j(\sigma)$}{$\pi_{j-1}(\sigma)$}  & $(\theta_{\rho_-(j-1)(j-1)},\theta_{j(j-1)})$ & $(\theta_{j(j-1)},\theta_{\rho_+(j-1)(j-1)})$\\
  \hline \\
 $(\theta_{\rho_-(j)j},\theta_{(j+1)j})$ & $\gamma_j\theta_{\rho_-(j)j}<\Lambda_j(\tau_{j-1}^-)$ & $\gamma_j\theta_{\rho_-(j)j}<\Lambda_j(\tau_{j-1}^+)$ \\ \\
 $(\theta_{(j+1)j},\theta_{\rho_+(j)j})$ & $\Lambda_j(\tau_{j-1}^-) < \gamma_j\theta_{\rho_+(j)j}$ & $\Lambda_j(\tau_{j-1}^+)<\gamma_j\theta_{\rho_+(j)j}$
 \end{tabular}
 \end{center}
 
 Here  $s_d+\ell_d+1$ is identified with $s_d+1$ and $(s_d+1)-1$ is identified with $s_d+\ell_d$. 
\end{proposition}

\begin{proof}

 Let $\Phi^d$ denote the flow direction map for the cyclic feedback system $\SWITCH^d(Z;\tau)$. We show that $\Phi^d(\sigma) = 0$ if and only if the conditions are satisfied. Applying Theorem \ref{thm: FP} then completes the proof. 
 
  We first assume that 
  $\sigma$ is the singular loop characteristic cell of $\SWITCH^d(Z,\tau)$. Then $\sigma = \sC^d(\sigma;\tau)$. Since $\sigma \in \Eq^d(\tau)$, $\sC^d(\sigma;\tau)$ is an equilibrium cell of $\SWITCH^d(Z;\tau)$ so that 
 \[
    \Phi^d(\sigma) = \Phi^d(\sC^d(\sigma;\tau))=0.
 \]

 If $\sigma$ is not the singular loop characteristic cell then $\sigma$ must be  a regular cell of $\SWITCH^d(Z;\tau)$ by Lemma \ref{lem: inessential cfs} and \ref{lem: essential cfs}. Since $\sigma\in \cN(\tau)$, 
 \[
 \pi_j(\sigma) = (\theta_{(j+1)j}, \theta_{\rho_+(j)j}) \quad  \mbox{ or } \quad  \pi_j(\sigma) = (\theta_{\rho_-(j)j},\theta_{(j+1)j})
 \]
 for each $j$. 
 
 Fix $j$ and suppose $\sigma$ satisfies  $\pi_{j}(\sigma) = (\theta_{(j+1)j}, \theta_{\rho_+(j)j})$.
 %$\pi_{j-1}(\sigma) = (\theta_{j(j-1)}, \theta_{\rho_+(j-1)(j-1)})$. 
  Since $\sigma\in \Eq^d(\tau)$, we have $\Phi_j^d(\sC(\sigma;\tau)) = 0$. Since $\pi_j(\sigma)$ and $\pi_j(\sC(\sigma;\tau))$ have the same left boundary  (by the assumption on $\sigma$), we have 
 \[
   1 = \cL^d(\sC(\sigma;\tau),j,-;0) = \cL^d(\sigma,j,-;0). 
 \]
 Therefore, $\Phi^d_j(\sigma) = 0$  if and only if $\cL^d(\sigma,j,+;0) = -1$. This is equivalent to
 \begin{align}\label{eq: Lambda(sigma) ineq}
     \sgn(-\gamma_j\theta_{\rho_+(j)j} + \Lambda_j(\sigma)) = -1, \quad \mbox{ or } \quad 
     \Lambda_j(\sigma) < \gamma_j\theta_{\rho_+(j)j}.
 \end{align}
 If $\pi_{j-1}(\sigma) = (\theta_{\rho_-(j-1)(j-1)},\theta_{j(j-1)})$ then $\sigma \in \cN(\tau_{j-1}^-)$ and if $\pi_{j-1}(\sigma) = (\theta_{j(j-1)},\theta_{\rho_+(j-1)(j-1)})$ then $\sigma \in \cN(\tau_{j-1}^+)$. 
 Notice that $j-1$ is the only singular direction of $\tau$ that maps to $j$ under $\rho$ and $j-1$ is not a singular direction of $\tau_{j-1}^-$ or $\tau_{j-1}^+$. Therefore, by Lemma \ref{lem: top cells}, $\Lambda_j(\sigma) = \Lambda_j(\tau_{j-1}^-)$ or $\Lambda_j(\sigma) = \Lambda_j(\tau_{j-1}^+)$ according to value of $\pi_{j-1}(\sigma)$. Together with (\ref{eq: Lambda(sigma) ineq}), this observation proves the proposition in the case $\pi_j(\sigma) = (\theta_{(j+1)j}, \theta_{\rho_+(j)j})$. A similar argument applies to the case $\pi_j(\sigma) = (\theta_{\rho_-(j)j},\theta_{(j+1)j)})$.  
\end{proof}

\red{See Example \ref{ex: candidate eq cells} for an example application of Proposition \ref{prop: d-consistent}.} The next theorem shows that 
any equilibrium cell in $\cN(\tau)$ can be written as a product over consistent candidate equilibrium cells. 
\begin{theorem}\label{thm:decomp}
    Let $\kappa\in\cN(\tau)$. Then  $\kappa$ is an equilibrium cell of $\SWITCH(Z)$ if and only if
    \begin{itemize}
        \item $\kappa\in\Eq(\tau)$, and
\item for every $d$, $\kappa^d$ is a $d$-consistent candidate equilibrium cell.
    \end{itemize}
\end{theorem}

\begin{proof}
 Let $\kappa\in\cN(\tau)$. By Lemma \ref{lem: Eq}, $\kappa\in\Eq(\tau)$ is a necessary condition for $\kappa$ to be an equilibrium cell. 
 We may therefore assume $\kappa \in \Eq(\tau)$. 
 
    By \red{Theorem} \ref{thm: FP}, $\kappa$ is an equilibrium cell if and only if $\Phi_j(\kappa) = 0$ for every $j\in V$. Lemma \ref{lem: Phi and Phid equal} shows that if $r$ is a regular direction of $\tau$ then $\Phi_r(\kappa) = \Phi_r^{n+1}(C(\kappa;\tau);\tau)$ and if $s\in \{s_d +1,\ldots, s_d+\ell_d\}$ is a singular direction of $\kappa$ then  $\Phi_s(\kappa) = \Phi_s^d(C(\kappa;\tau);\tau)$. 
    It remains to show that $\Phi_j(\kappa) = 0$ if and only if $\kappa^d$ is $d$-consistent when $j\in\{s_d+1,\ldots,s_d+\ell_d\}$ is a regular direction of $\kappa$ but singular direction of $\tau$. There are four cases determined by the values of $\pi_j(\kappa)$ and $\pi_{j-1}(\kappa)$ indicated by the table in the statement of Proposition \ref{prop: d-consistent}. We prove the case $\pi_j(\kappa) = (\theta_{\rho_-(j)j},\theta_{\rho(j)j})$ and $\pi_{j-1}(\kappa) = (\theta_{\rho_-(j-1)(j-1)},\theta_{j(j-1)})$. The remaining cases are similar. Since $\pi_{j-1}(\kappa) = (\theta_{\rho_-(j-1)(j-1)},\theta_{j(j-1)})$ we have $\Lambda_j(\kappa) = \Lambda_j(\tau_j^-)$ so that 
    \[
        \cL(\kappa,j,-) = \sgn(-\gamma_j\theta_{\rho_-(j)j} + \Lambda_j(\tau_j^-)). 
    \]
 Therefore $\cL(\kappa,j,-) = 1$ if and only if $\kappa^d$ is $d$-consistent. Since $\kappa^d\in\Eq^d(\tau)$,
 \[
 \cL^d(C(\kappa^d;\tau),j,-) = \sgn(-\gamma_j\cdot 0 + \Lambda_j^d(C(\kappa^d;\tau))) = 1 = -\cL^d(C(\kappa^d;\tau),j,+). 
 \]
 But $\cL^d(C(\kappa;\tau),j,+) = \sgn(-\gamma_j\theta_{\rho(j)j} + \Lambda_j^d(C(\kappa;\tau)) = \cL(\kappa,j,+)$ since $\Lambda_j^d(C(\kappa;\tau)) = \Lambda_j(\kappa)$. Therefore $\Phi_j(\kappa) = 0$ if and only if $\kappa^d$ is $d$-consistent, completing the proof of the theorem. 
 
\end{proof}

\red{
\begin{example}\label{ex: candidate eq cells}
    % Consider the positive toggle plus system of Example \ref{ex: positive toggle} whose cell complex and labeling map are pictured in Figure \ref{fig: neighbors}(b). For the loop characteristic cell $\tau = \{\theta_{21}\}\times \{\theta_{12}\}$, the cycle decomposition is given by $\rho = (c_1)$ where $c_1(1) = 2$ and $c_1(2) = 1$. The cell neighborhood of $\tau$, $\cN(\tau)$ is shaded gray in Figure \ref{fig: cones}(a). For $x\in \cN(\tau)$, $x_1<\theta_{11}$ and $x_2<\theta_{22}$ so we have 
    % \begin{align*}
    %     \dot{x}_1 &= -\gamma_1 x_1 + \Lambda_1(x) = -\gamma_1 x_1 + L_{11}\sigma_{12}(x_2) \\
    %     \dot{x}_2 &= -\gamma_2 x_2 + \Lambda_2(x) = -\gamma_2 x_2 + L_{22}\sigma_{21}(x_1),
    % \end{align*}
    % i.e. $\Lambda^1(x^1;\tau) = (L_{11}\sigma_{12}(x_2),L_{22}\sigma_{21}(x_1))$ with $x^1 = (x_1,x_2)$. The dynamics for the CFS $\SWITCH^1(Z;\tau)$ is given by $\dot{x} = -\Gamma^1 x^1 + \Lambda^(x^1;\tau)$. The cell complex, which consists of cones rooted in $\tau$, and labeling map for $\SWITCH^1(Z;\tau)$ is pictured in Figure \ref{fig: cones}(b). The figure indicates that $C^1(\kappa_3;\tau) = C(\kappa_3;\tau)$ is a 1-candidate equilibrium cell of $\tau$. Since $\kappa_3 = C(\kappa_3;\tau)$, the equilibrium in $C(\kappa_3;\tau)$ is contained in $\kappa_3$ so that $\kappa_3$ is a $1$-consistent candidate equilibrium cell and therefore is an equilibrium cell of $\SWITCH(Z)$ by Theorem \ref{thm:decomp}. 
    
     Consider the decomposition of the positive toggle plus system described in Example \ref{ex: decomp}. The cell complexes $\chi^1(\til{\tau})$ and $\chi^2(\til{\tau})$ with corresponding labeling maps $\cL^1$ and $\cL^2$ are pictured in Figure \ref{fig: cones}(c) and (d). $\SWITCH^1(Z;\til{\tau})$ admits three $1$-candidate equilibrium cells given by $\til{\tau}$, $\til{\tau}_1^-$ and $\til{\tau}_1^+$. Applying Proposition \ref{prop: d-consistent}, we have $\til{\tau}$ is a consistent equilibrium cell vacuously, $\til{\tau}_1^-$ is not a consistent equilibrium cell since $\Lambda_1(\til{\tau}_1^-) = L_{11}U_{12}<\gamma_1\theta_{21}$, and $\til{\tau}_1^+$ is a consistent equilibrium cell since $\Lambda_1(\til{\tau}_1^+) = U_{11}U_{12} < \gamma_1\theta_{\infty 1} = \infty$. $\SWITCH^2(Z;\til{\tau})$ admits one $2$-candidate equilibrium cell, $\til{\tau}_2^-$. Since $\Lambda^2(\til{\tau}_2^-) = L_{22}U_{21} > \gamma_2 \theta_{21}$, it is a consistent equilibrium cell. By Theorem \ref{thm:decomp}, the equilibrium cells of $\SWITCH(Z)$ which are contained in $\cN(\tau)$ are given by a product over projection of the consistent equilibrium cells, namely $\til{\tau}_2^- = \pi_1(\til{\tau}) \times \pi_2(\til{\tau}_2^-)$ and $\kappa_4 = \pi_1(\til{\tau}_1^+) \times \pi_2(\til{\tau}_2^-)$. 
\end{example}
}

\subsection{Stability of Equilibria}\label{sec: gen net stability}

The decomposition in Theorem~\ref{thm:decomp}
implies that the Jacobian $J(\eps)$ for $\cS(Z,\eps)$ has a block structure up to order $\eps$. In particular, we have that there is a $B\in \bR^{N\times N}$ such that
\begin{align*}
    J(\eps) = \left(
    \begin{array}{ccccc}
     J^1(\eps) \\ 
     & J^2(\eps) \\
     & & \ddots \\
     & & & J^n(\eps) \\
     & & & & J^{n+1}(\eps)
    \end{array}
    \right) + \eps B, \qquad 
     J^{n+1} = \left(
     \begin{array}{cccc}
        -\gamma_{s_n + 1}  \\
        & -\gamma_{s_n + 2} \\
        & & \ddots \\
        & &  & -\gamma_{N}
    \end{array}
    \right),
\end{align*}
and $J^d(\eps)$ for $1\leq d \leq n$ has the same structure as \eqref{eq: CFS Jacobian} although the entries $\sigma_{(j+1)j}'$ are replaced with $\frac{\partial}{\partial x_j} \Lambda_{j+1}(x;\eps)$. The block $J^d(\eps)$ is the Jacobian for the sigmoidal system $\cS^d(Z,\eps;\tau)$ obtained from perturbing the switching system $\SWITCH^d(Z;\tau)$. This implies that the stability of an equilibrium $x^\eps$ associated to an equilibrium cell $\kappa\in \cN(\tau)$ is determined by the stability of $\kappa^d$ as an equilibrium cell of $\SWITCH^d(Z;\tau)$ for $d=1,\ldots, n+1$. We formally state this observation as a theorem after the following definition.  

\begin{definition}
 An equilibrium cell $\kappa\in \cN(\tau)$ is \emph{$d$-stable} if $\kappa^d$ is stable as a cell of $\SWITCH^d(Z;\tau)$, and \emph{$d$-unstable} otherwise. 
\end{definition}

\begin{theorem}\label{thm: d-stability}
 An equilibrium cell $\kappa\in\cN(\tau)$ is stable if and only if $\kappa$ is $d$-stable for each $d$. 
\end{theorem}

 As a result of Theorem \ref{thm: d-stability}, we can immediately generalize the stability results given by Propositions \ref{prop: reg cell stable}, \ref{prop: pos sing cell unstable}, \ref{prop: neg N<3 stable}, and \ref{prop: neg sing cell stability}, which results in the following theorem. 
 
 \begin{theorem}
 Let $Z = (L,U,\theta,\Gamma)$ be a switching parameter, $\tau$ be a loop characteristic cell, and $\kappa\in \cN(\tau)$ be an equilibrium cell. Then $d$-stability of $\kappa$ can be determined as follows. 
 \begin{enumerate}
     \item If $\kappa^d$ is a regular cell then $\kappa$ is $d$-stable. 
     \item If $\kappa^d$ is a singular cell and $c_d$ is positive, then $\kappa$ is $d$-unstable. 
     \item If $\kappa^d$ is a singular cell, $c_d$ is negative, and $\ell_d \leq 2$, then $\kappa$ is $d$-stable. 
     \item If $\Gamma = I$, $\kappa^d$ is a singular cell, $c_d$ is negative, and $\ell_d > 2$, then $\kappa$ is $d$-unstable. 
 \end{enumerate}
 \end{theorem}
 
 We remark that an equilibrium cell $\kappa$ is always $d=n+1$-stable since the only equilibrium cell of $\SWITCH^{n+1}(Z;\tau)$ is regular. 

\red{
\begin{example}
    Consider the decomposition of the positive toggle plus system at $\til{\tau}$ pictured in Figure \ref{fig: cones}. The equilibrium cell $\kappa_4$ is $1$-stable because $\pi_1(\kappa_4) = C^1(\kappa_4;\til{\tau}) = C^1(\til{\tau}_1^+;\til{\tau})$ is stable in $\SWITCH^1(Z;\til{\tau})$ as can be seen in Figure \ref{fig: cones}(d). $\kappa_4$ is also $2$-stable because $\pi_2(\kappa_4) = C^2(\kappa_4;\til{\tau}) = C^2(\til{\tau}_2^-;\til{\tau})$ is stable in $\SWITCH^2(Z;\til{\tau})$ as can be seen in Figure \ref{fig: cones}(c). The equilibrium cell $\til{\tau}_2^-$ is $1$-unstable because $\pi_1(\til{\tau}_2^-) = \til{\tau}^1$ is unstable in $\SWITCH^1(Z;\til{\tau})$. 
\end{example}
}
\section{Proof of Theorem \ref{thm: FP}}\label{sec: eq cell proofs}

Before proving the theorem, we will prove some technical lemmas.
We begin by proving Lemma \ref{lem: top cells}.

\begin{proof}[Proof of Lemma \ref{lem: top cells}.]
 \red{Let $\tau\in \chi$ be a cell, $\kappa \in \cN(\tau)$ be a cell in the cell neighborhood $\cN(\tau)$ of $\tau$ (see Definition \ref{defn: cell neighborhood}), and $(j,i)\in E$ be an edge. }  Suppose $j$ is a regular direction of $\tau$ with $\pi_j(\tau) = (\theta_{i_1j},\theta_{i_2 j})$. Then $\sigma_{ij}(\tau)$ is well defined. Since  $\tau\subset \ol{\kappa}$, 
 \[
 (\theta_{i_1j},\theta_{i_2j}) = \pi_j(\tau)\subset\pi_j(\ol{\kappa}) = [\theta_{i_1'j},\theta_{i_2'j}].
 \] 
  Because $\theta_{i_1'j}<\theta_{i_2'j}$ are consecutive thresholds  we must have $\theta_{i_1j} = \theta_{i_1'j}$ and $\theta_{i_2j} = \theta_{i_2'j}$.  Therefore $\pi_j(\kappa) = \pi_j(\tau)$. This implies $\sigma_{ij}(\kappa) = \sigma_{ij}(\tau)$. 
  
  Now suppose $j\in\sd(\tau)$, and $\pi_j(\tau) = \{\theta_{i_0j}\}$.  Assume $i_0\neq i$. This implies that $\sigma_{ij}(\theta_{i_0 j})$ is defined and $\sigma_{ij}(\tau) = \sigma_{ij}(\theta_{i_0 j})$.  Let $\theta_{i_1 j}<\theta_{i_0j}<\theta_{i_2j}$ be consecutive thresholds. Then either $\pi_j(\kappa) = (\theta_{i_1 j},\theta_{i_0j})$, $\pi_j(\kappa) = (\theta_{i_0j},\theta_{i_2j})$, or $\pi_j(\kappa) = \{\theta_{i_0j}\}$. Since  $\theta_{ij}\notin (\theta_{i_1j},\theta_{i_2j})$, $\sigma_{ij}(x_j)=\sigma_{ij}(\tau)$ for all $x_j\in(\theta_{i_1j},\theta_{i_2 j})$ and $\sigma_{ij}(\kappa) = \sigma_{ij}(\tau)$. 
  
  Since $\Lambda_i$ is a product of sums of switching functions $\sigma_{ij}$,  $\Lambda_i(\kappa)$ is well defined whenever $\sigma_{ij}(\tau)$ is well defined for all $j\in\Sources(i)$. That is, $\Lambda_i(\kappa)$ is well defined whenever $i\notin \rho(\sd(\tau)) = \{\rho(j)\;|\; j\in \sd(\tau)\}$. 
\end{proof}

%%%%%%%%%%%%%%%%

 \red{Recall that $\tau_s^\pm$ denotes an $s$-neighbor of a cell $\tau\in \chi$ (see Definition \ref{defn: neighbors}).} A consequence of Lemma \ref{lem: top cells} is that the labeling map $\cL$ and thus the flow direction map $\Phi$ \red{(see Definition \ref{defn: flow})}, are well defined because $\Lambda_{\rho(s)}(\tau_s^\pm)$ is well defined for every singular direction $s$ of a loop characteristic cell $\tau$.

\begin{lemma}\label{lem: sd labels}
 Let $\tau\in \chi$ be a loop characteristic cell. If $s$ is a singular direction of $\tau$ then $\Lambda_{\rho(s)}(\tau_s^\pm)$ is well defined. 
\end{lemma}

\begin{proof}
 Since $s$ is the unique singular direction of $\tau$ which maps to $\rho(s)$, we have $\rho(s)\notin \rho^{\tau_s^\pm}(\sd(\tau_s^\pm))$ since $s$ is a regular direction of $\tau_s^\pm$. Lemma \ref{lem: top cells} then implies $\Lambda_{\rho(s)}(\tau_s^\pm)$ is well defined. 
\end{proof}

Having confirmed that $\cL$ and $\Phi$ are well defined on their domain, we now prove Theorem \ref{thm: FP}.

\subsection{Proof of Theorem \ref{thm: FP}}

Let $\tau\in\chi$ be an equilibrium cell and $x^{\eps}$ be an equilibrium of $\cS(Z,\eps)$ such that $x^\eps\to x^* \in\tau$ as $\eps\to 0^+$. Let $\rho = \rho^\tau$. 

First suppose that statement (1) does not hold. Then there is a singular direction $s\in\sd(\tau)$ so that $\rho(j)\neq s$ for all $j\in \sd(\tau)$. By Lemma \ref{lem: top cells}, $\Lambda_s(\tau)$ is well defined and  $\Lambda_s(\kappa) = \Lambda_s(\tau)$ for all $\kappa \in \cN(\tau)$. Therefore, for all $x\in \cN(\tau)$, $\Lambda_s(x;\eps)\to\Lambda_s(\tau)$  as $\eps \to 0$. Since $x^\eps\to x^* \in \tau$, there is an $A>0$ so that for  $\eps<A$, $x^\eps\in\cN(\tau)$. Therefore, $\Lambda_s(x^\eps,\eps)\to \Lambda_s(\tau)$ as $\eps \to 0$. Finally, since $\pi_s(\tau) = \{\theta_{\rho(s)s}\}$, we have $x_s^\eps\to \theta_{\rho(s)s}$ so that the $s$th component of $x^*$ satisfies $x^*_s =\theta_{\rho(s)s}$. These arguments imply  that taking the limit of the expression 
\[ \lim_{\eps\to 0^+} -\gamma_s x_s^{\eps} + \Lambda_s(x^{\eps};\eps) = 0 \]
results in 
\[ -\gamma_s \theta_{\rho(s) s} + \Lambda_s(\tau)= 0.\]
 But this contradicts the fact that $Z$ is a regular parameter.
This proves (1). 

To prove (2) we consider regular and singular directions separately.
 For a regular direction, $r$, let $x_r^* = \lim_{\eps\to 0} x_r^{\eps}$. We have
\begin{align*}
\lim_{\eps\to 0^+} -\gamma_r x_r^{\eps} + \Lambda_r(x^{\eps};\eps) = -\gamma_r x_r^{*} + \Lambda_r(\tau) = 0. 
\end{align*}
Since $x_r^*\in (\theta_{a_r r},\theta_{b_r r})$,  we must have $-\gamma_r \theta_{a_r r} + \Lambda_r(\tau) >0$ and $-\gamma_r \theta_{b_r r} + \Lambda_r(\tau) < 0$. Therefore, $\Phi_r(\tau) = 0$. 

For a singular direction, $s$,
 let  $\eta = \min\left\{\frac{\theta_{\rho_+(s) s} - \theta_{\rho(s)s}}{2},\frac{\theta_{\rho(s)s}-\theta_{\rho_-(s) s}}{2}\right\}$. Note that $\eta$ is independent of $\eps$. With this $\eta$, define the sample points $x^{\pm \eta,\eps}$ by 
\begin{align*}
 x^{\pm \eta,\eps}_j = \begin{cases}
                       x_j^{\eps}, \quad & j\neq s \\
                       x_s^{\eps}\pm \eta,  \quad & j = s.
                      \end{cases}
\end{align*}
Note that 
$x^{\pm\eta,\eps}\to \tau_s^\pm$ as $\eps\to 0^+$. Using the fact that $x^\eps$ is an equilibrium of the sigmoid system $S(Z,\eps)$, we have 
\begin{align*}
  0 =& -\gamma_{\rho(s)}x_{\rho(s)}^{\eps} + \Lambda_{\rho(s)}(x^{\eps};\eps) \\
  =& -\gamma_{\rho(s)}x_{\rho(s)}^{\eps} + \Lambda_{\rho(s)}(x^{\eps};\eps) - \Lambda_{\rho(s)}(x^{\pm \eta,\eps};\eps) + \Lambda_{\rho(s)}(x^{\pm \eta,\eps};\eps)\\
  =& -\gamma_{\rho(s)}x_{\rho(s)}^{\eps} + \left(\prod_{\substack{\ell=1 \\ s\notin I_\ell}}^{p_s} \sum_{j\in I_{\ell}} \sigma_{\rho(s)j}(x_j^{\eps};\eps)\right)(\sigma_{\rho(s)s}(x_s^{\eps};\eps) - \sigma_{\rho(s)s}(x_s^{\pm\eta,\eps};\eps)) + \Lambda_{\rho(s)}(x^{\pm \eta,\eps};\eps).
\end{align*}
Rearranging the resulting equality, we get
\begin{align}\label{eq: Phi_s = 0}
 \sgn(-\gamma_{\rho(s)}x_{\rho(s)}^{\eps}+ \Lambda_{\rho(s)}(x^{\pm \eta,\eps};\eps)) &= -\sgn\left((\sigma_{\rho(s)s}(x_s^{\eps};\eps) - \sigma_{\rho(s)s}(x_s^{\pm\eta,\eps};\eps))\right)
\end{align}
 The left hand side of \eqref{eq: Phi_s = 0} is $\cL(\tau,s,\pm)$ in the limit of $\eps\to 0$. To see this notice that since $\rho$ is a permutation, there is no $j\in\sd(\tau_{s}^\pm)$ so that $\rho^\tau(j) = \rho^\tau(s)$. By Lemma \ref{lem: top cells}, for every $x\in \cN(\tau_s^\pm)$, the function values $\Lambda_{\rho(s)}(x;\eps)\to\Lambda_{\rho(s)}(\tau_s^\pm)$ as $\eps \to 0^+$.  Since in the same limit $x^{\pm\eta,\eps}\to \tau_s^\pm$, there is an $A>0$ such  that for $\eps<A$,  we have $x^{\pm\eta,\eps}\in\cN(\tau_s^\pm)$. Therefore, $\Lambda_{\rho(s)}(x^{\pm\eta,\eps};\eps)\to \Lambda_{\rho(s)}(\tau_s^\pm)$ as $\eps \to 0^+$. Because the projection $\pi_{\rho(s)}(\tau) = \{\theta_{\rho^2(s)\rho(s)}\}$ we conclude that $x_{\rho(s)}^\eps\to \theta_{\rho^2(s)\rho(s)}$.
  To summarize, we have shown  
\begin{align}\label{eq: limit label map}
 \lim_{\eps\to 0^+} \sgn( -\gamma_{\rho(s)}x_{\rho(s)}^{\eps}+ \Lambda_{\rho(s)}(x^{\pm \eta,\eps};\eps)) = \sgn(-\gamma_{\rho(s)} \theta_{\rho^2(s)\rho(s)} + \Lambda_{\rho(s)}(\tau_s^{\pm})) = \cL(\tau,s,\pm).
\end{align}

To show that $\cL(\tau,s,+) = - \cL(\tau,s,-)$, we study the right hand side of \eqref{eq: Phi_s = 0}. By Properties (1) and (4) of sigmoidal perturbations, there is a neighborhood $\cU\subset \bR_+$ of $\theta_{\rho(s)s}$ such that $\sigma_{\rho(s)s}(\cdot;\eps)$ is  monotone increasing on $\cU$ when $\bs_{\rho(s)s}=1$ and decreasing when $\bs_{\rho(s)s}=-1$. Since $\sigma_{\rho(s)s}$ is monotone non-increasing or non-decreasing according to $\bs_{\rho(s)s}$, this implies  
\begin{align*}
    -\sgn\left((\sigma_{\rho(s)s}(x_s^\eps;\eps)-\sigma_{\rho(s)s}(x_s^{\pm\eta,\eps};\eps)\right) &=  -\sgn\left(x_s^\eps-x_s^{\pm\eta,\eps} \right)\bs_{\rho(s)s}  \\
    &=-(\mp 1) \bs_{\rho(s)s}= \pm \bs_{\rho(s)s}. 
\end{align*}

We have shown
\begin{align*}
    \cL(\tau,s,\pm) &= \lim_{\eps\to 0^+}\sgn( -\gamma_{\rho(s)}x_{\rho(s)}^{\eps}+ \Lambda_{\rho(s)}(x^{\pm \eta,\eps};\eps)) \\
    &= \lim_{\eps\to 0^+}-\sgn\left((\sigma_{\rho(s)s}(x_s^{\eps};\eps) - \sigma_{\rho(s)s}(x_s^{\pm\eta,\eps};\eps))\right) 
    = \pm \bs_{\rho(s)s}
\end{align*}
where the first equality follows from \eqref{eq: limit label map} and the second equality follows from \eqref{eq: Phi_s = 0}. 
Therefore $\cL(\tau,s,+) = - \cL(\tau,s,-)$ so that $\Phi_s(\tau) = 0$. 
This finishes the proof of (2) and thus the forward implication in statement (a). We now proceed with the backward implication in (a).

%%%%%%%%%%%%%%%%%%%%%%%%%%%%%%%%%%%%%%%%%%%%%%%%%%%%%
\red{
For a singular direction $s\in\sd(\tau)$, let $\cU_{2,s}(\eps)$ be the neighborhood of $\theta_{\rho(s)s}$ defined for $\sigma_{\rho(s)s}$ in property 4 of sigmoidal perturbations (Definition \ref{defn: sigmoid}). For a regular direction $r\notin \sd(\tau)$, we have $\Phi_r(\tau) = 0$ or 
\begin{align*}
    \sgn(-\gamma_r\theta_{a_r r} + \Lambda_r(\tau)) = -\sgn(-\gamma_r\theta_{b_r r} + \Lambda_r(\tau)). 
\end{align*}
Since $\Lambda_r(x;\eps)\to \Lambda_r(\tau)$ for $x\in \tau$, we may choose $\eta_r>0$ small enough so that 
\begin{align}\label{eq: reg sgn}
\sgn(-\gamma_r(\theta_{a_r r}+\eta_r) + \Lambda_r(\tau)) = -\sgn(-\gamma_r(\theta_{b_r r}-\eta_r) + \Lambda_r(\tau)). 
\end{align}

We now define a closed neighborhood $\tau(\eps)$ of $\tau$ by 
\[
    \pi_j(\tau(\eps)):= \begin{cases}
        \ol{\cU_{2,j}(\eps)}, \quad &j\in \sd(\tau) \\
        [\theta_{a_r r} + \eta_r,\theta_{b_r r} - \eta_r], \quad &j \notin \sd(\tau). 
    \end{cases}
\]
Note that $\tau\subset \tau(\eps)$. By properties 3 and 4 of sigmoidal perturbations we may choose $A\in \bR^{N\times N}$ so that for $\eps < A$ and $x\in \tau(\eps)$ 
\begin{align}\label{eq: derivatives}
    \left|\frac{\partial}{\partial x_s}\Lambda_{\rho(s)}(x;\eps)\right|\geq 2\gamma_{\rho(s)}, \quad \mbox{and} \quad \left|\frac{\partial}{\partial x_r} \Lambda_{r}(x;\eps)\right|\leq \tfrac{1}{2}\gamma_{r}. 
\end{align}
Property 4 implies that the image of $\Lambda_{\rho(s)}(\cdot;\eps)$ on $\tau(\eps)$ converges, as $\epsilon \to 0$, to the interval with endpoints $\Lambda_{\rho(s)}(\tau_s^-)$ and $\Lambda_{\rho(s)}(\tau_s^+)$, i.e. $\Lambda_{\rho(s)}(\tau(\eps);\eps)\to (\Lambda_{\rho(s)}(\tau_s^-),\Lambda_{\rho(s)}(\tau_s^+))$. Here  we have assumed without loss of generality that $\Lambda_{\rho(s)}(\tau_s^-)<\Lambda_{\rho(s)}(\tau_s^+)$. Since $\Phi_s(\tau) = 0$, we have $\gamma_{\rho(s)}\theta_{\rho^2(s)\rho(s)}\in (\Lambda_{\rho(s)}(\tau_s^-),\Lambda_{\rho(s)}(\tau_s^+))$. Therefore, we may further refine our choice of  $A$ so that $\gamma_{\rho(s)}\theta_{\rho^2(s)\rho(s)} \in \Lambda_{\rho(s)}(\tau(\eps);\eps)$ for each $\eps < A$. That is, for each $s\in \sd(\tau)$, there is an $x\in \tau(\eps)$ so that 
\[
    f_{\rho(s)}:=-\gamma_{\rho(s)}x_{\rho(s)} + \Lambda_{\rho(s)}(x;\eps) = \dot{x}_{\rho(s)}|_{x} = 0. 
\]
For a regular direction $r$, we also have that there is an $x\in \tau(\eps)$ with  
\[
    f_r:=-\gamma_{r}x_r + \Lambda_{r}(x;\eps) = \dot{x}_{r}|_x = 0
\]
by \eqref{eq: reg sgn}. 

 By \eqref{eq: derivatives}, if $x\in \tau(\eps)$ and $s\in \sd(\tau)$ is a singular direction or $r$ is a regular direction then
\[
    \left|\frac{\partial f_{\rho(s)}}{\partial x_s}(x)\right| > \gamma_{\rho(s)}>0 \quad \mbox{and} \quad \frac{\partial f_{r}}{\partial x_r}(x) < -\tfrac{1}{2}\gamma_{r}<0. 
\]
%[The old argument is commented out in the latex file.  ]

}

In either case, the derivatives are bounded away from zero so we may apply the implicit function theorem to get a function
\[
 X_j:\prod_{i\neq j} \pi_i(\red{\tau(\eps))} \to \pi_j(\red{\tau(\eps)})
\]
so that 
\[
\red{f_{\rho(j)}} ((x_1,\ldots,X_j(x_1,\ldots,x_{j-1},x_{j+1},\ldots,x_N),\ldots,x_N)) = 0.
\]

Define $g:\red{\tau(\eps)} \to \red{\tau(\eps)}$ by 
\begin{align*}
    g(x_1,\ldots,x_N) := (X_1(x_2,\ldots,x_N),\ldots,X_N(x_1,\ldots,x_{N-1})).
\end{align*}
By Brouwer's fixed point theorem, $g$ has a fixed point $x$ in $\red{\tau(\eps)}$ so that we can simultaneously solve $\dot{x}_j = 0$ for all $j$. 

To prove part (b) and show that
the fixed point is unique, we apply the \red{inverse} function theorem. Let $x^\eps$ be an equilibrium of $\cS(Z,\eps)$. We note that except for $O(\eps)$ terms, the Jacobian $J(\eps)$ has a block diagonal structure with blocks $J^d(\eps)$ corresponding to the cyclic feedback decomposition of $\SWITCH(Z)$ at $\tau$ (see Section \ref{sec: decomposition}). Each block corresponding to singular directions of $\tau$ has the form of \eqref{eq: CFS Jacobian} while the block corresponding to regular directions is diagonal. Therefore the determinant is given by 
\begin{align}
    \det(J(x;\eps)) = \prod_{d=1}^n(-1)^{\ell_d}\left( \prod_{s=s_d+1}^{s_d + \ell_d} \gamma_s - \sgn(c_d)M^d(x,\eps)\right)\left(\prod_{r=s_n +1}^N -\gamma_r\right) + O(\eps)
\end{align}
where $M^d(x,\eps) = \prod_{s=s_d+1}^{s_d + \ell_d} |\sigma'_{\rho(s)s}(x_s,\eps)|$. Since $M^d(x,\eps) \to \infty$ for $x\in \tau$ we have that for small enough $\eps$, $J(x;\eps)$ is non-singular on $\red{\tau(\eps)}$. The inverse function theorem then implies that the equilibrium $x^\eps$ of $\cS(Z,\eps)$ in $\red{\tau(\eps)}$ is unique.

%%%%%%%%%%%%%%%%%%%%%%%%%%%%%
 \section{Proof of Cyclic Feedback System Results}\label{sec: cfs results proofs}

In this section, we prove the results of Section \ref{sec: cfs results}. The section is organized into subsections which correspond to the subsections of Section \ref{sec: cfs results}. \red{Recall that the edges of a cyclic feedback network (CFN) are of the form $(j,j+1)$ where $j+1$ is computed modulo $N$ (see Definition \ref{defn: cfn}). \red{Furthermore,  by changing variables} we may assume without loss of generality 
%\Tomas{ Do we argue somewhere why this is without loss of generality?}
that all edges of a positive CFN are activating and the edges of a negative CFN are all activating except for the edge $N\dashv 1$ which is repressing \red{(see Section \ref{sec: cfs equilibria})}. }

\subsection{Equilibrium Cells}\label{sec: cfn eq cell proofs}

\begin{proof}[Proof of Lemma \ref{lem: inessential cfs}]
Let $\kappa'$ be a regular equilibrium cell. We will show that $\kappa' = \kappa$ \red{where $\kappa$ is the claimed unique equilibrium cell defined in the statement of the lemma}. \red{Recall that for each $j\in V$, $\theta_{-\infty j} = 0$ and $\theta_{\infty j} = \infty$ (see Definition \ref{defn: cell complex}). We have \sout{First note that}}
 \begin{align*}
     \sgn(-\gamma_j\theta_{-\infty j} + \Lambda_j(\kappa')) &=\sgn(0+\Lambda_j(\kappa')) = 1, \mbox{ and }\\
     \sgn(-\gamma_j\theta_{\infty j} + \Lambda_j(\kappa')) &= \sgn(-\infty +\Lambda_j(\kappa')) = -1. 
 \end{align*}
 If $j$ is inessential and $\gamma_j\theta_{(j+1)j}<L_{j(j-1)}$ then  
\[\sgn(-\gamma_j\theta_{(j+1)j} + \Lambda_j(\kappa')) = -1.\]
If $j$ is inessential and $U_{j(j-1)}<\gamma_j\theta_{(j+1)j}$ then 
\[\sgn(-\gamma_j\theta_{(j+1)j} + \Lambda_j(\kappa')) = 1.\]
In either case, $\Phi_j(\kappa') = 0$ if and only if $\pi_j(\kappa') = \pi_j(\kappa)$. 

Now suppose that $k$ is essential and $k+1$ is inessential. If $\pi_k(\kappa') = (0,\theta_{(k+1)k})$, then, \red{using the assumption on the sign of the edges},
\begin{align*}
    \Lambda_{k+1}(\kappa') = \begin{cases}
     L_{(k+1)k}, \quad & \sgn(\rho) = 1 \mbox{ or } k<N\\
     U_{(k+1)k}, \quad & \sgn(\rho) = -1 \mbox{ and } k=N
    \end{cases}
\end{align*}
This implies
\begin{align*}
    \sgn(-\gamma_{k+1}\theta_{(k+2)(k+1)}+\Lambda_{k+1}(\kappa')) = \begin{cases}
     -1, \quad & \sgn(\rho) = 1 \mbox{ or } k<N\\
     1, \quad & \sgn(\rho) = -1 \mbox{ and } k=N
    \end{cases}
\end{align*}
so that $\Phi_{k+1}(\kappa') = 0$ if and only if $\pi_{k+1}(\kappa') = \pi_{k+1}(\kappa)$. A similar argument shows $\pi_{k+1}(\kappa') = \pi_{k+1}(\kappa)$ when $\pi_k(\kappa') = (\theta_{(k+1)k},\infty)$. An induction argument then shows that for every essential node $j$, $\pi_j(\kappa') = \pi_j(\kappa)$. 

\red{We have shown that the only regular equilibrium cell is $\kappa$. Since equilibrium cells are a subset of loop characteristic cells, to complete the proof
 we need only show that the only singular loop characteristic cell $\tau = \prod \{\theta_{(j+1)j}\}$ is not an equilibrium cell. }  Let $j$ be a node so that $j+1$ is inessential. Assume $U_{(j+1)j}<\gamma_{j+1}\theta_{(j+2)(j+1)}$. Then 
\begin{align*}
    \cL(\tau,j,+) = \cL(\tau,j,-)=-1
\end{align*}
since $\sgn(-\gamma_{j+1}\theta_{(j+2)(j+1)}+\Lambda_{j+1}(\tau))= -1$. Similarly, 
\begin{align*}
    \cL(\tau,j,+) = \cL(\tau,j,-)=1
\end{align*}
when $\gamma_{j+1}\theta_{(j+2)(j+1)}<L_{(j+1)j}$. \red{This shows $\Phi_j(\tau)\neq 0$ so that $\tau$ is not an equilibrium cell.}
\end{proof}
%%%%%%%%%%%%%%%%%%%%%%%%%%%

\begin{proof}[Proof of Lemma \ref{lem: essential cfs}]
 \red{Let $\bRN$ be a cyclic feedback network and $Z$ be a switching parameter so that the corresponding cyclic feedback system $\SWITCH(Z)$ has no inessential nodes.  }First we show \red{that the singular loop characteristic cell,} $\tau\red{=\prod \{\theta_{(j+1)j}\}}$, is an equilibrium cell. \red{Using the assumption on the sign of the edges,} for every $j\in V$ we have
 \begin{align*}
     \Lambda_{j+1}(\tau_j^-) = \begin{cases}
        L_{(j+1)j}, \quad & \sgn(\rho) = 1 \mbox{ or } j<N \\
        U_{(j+1)j}, \quad & \sgn(\rho) = -1 \mbox{ and } j=N
     \end{cases}
 \end{align*}
 and
 \begin{align*}
      \Lambda_{j+1}(\tau_j^+) = \begin{cases}
        U_{(j+1)j}, \quad & \sgn(\rho) = 1 \mbox{ or } j<N \\
        L_{(j+1)j}, \quad & \sgn(\rho) = -1 \mbox{ and }j=N.
     \end{cases}
 \end{align*}
 Since $L_{(j+1)j}<\gamma_{j+1}\theta_{(j+2)(j+1)}<U_{(j+1)j}$, this implies $\Phi_j(\tau) = 0$. $\tau$ is therefore an equilibrium cell by Theorem \ref{thm: FP}. 

Now consider the case $\bRN$ is a positive CFN. Since $\pi_j(\kappa^L) = (0,\theta_{(j+1)j})$ and all edges are activating, $\Lambda_{j+1}(\kappa^L) = L_{(j+1)j}$. We therefore have
\begin{align*}
    \cL(\kappa^L,j+1,+) = \sgn(-\gamma_{j+1}\theta_{(j+2)(j+1)}+\Lambda_{j+1}(\kappa^L)) = -1. 
\end{align*}
A similar argument shows $\cL(\kappa^H,j+1,-) = 1$ for all $j$. Since 
\[
\sgn(-\gamma_{j+1}\theta_{-\infty j} +\Lambda_{j+1}(\kappa)) = 1 \quad \mbox{ and } \sgn(-\gamma_{j+1}\theta_{\infty j} + \Lambda_{j+1}(\kappa)) = -1
\]
for all cells $\kappa$, this implies $\Phi_{j+1}(\kappa^L) = \Phi_{j+1}(\kappa^H) = 0$ so that both cells are equilibrium cells. 

Let $\kappa$ be a regular cell different from $\kappa^H$ and $\kappa^L$. Then there is a $j$ so that $\pi_{j}(\kappa) = (\theta_{(j+1)j},\infty)$ but $\pi_{j+1}(\kappa) = (0,\theta_{(j+2)(j+1)})$. $\pi_{j}(\kappa) = (\theta_{(j+1)j},\infty)$ implies $\Lambda_{j+1}(\kappa) = U_{(j+1)j}$. But $\pi_{j+1}(\kappa) = (0,\theta_{(j+2)(j+1)})$ implies 
\[
\cL(\kappa,j+1,+) = \sgn(-\gamma_{j+1}\theta_{(j+2)(j+1)} + U_{(j+1)j}) = 1
\]
so that $\Phi_j(\kappa) = 1$ and $\kappa$ is not an equilibrium cell. 

Finally, consider the case that $\bRN$ is a negative CFN. Let $\kappa$ be a regular cell. Suppose there is a $j<N$ such that 
%$\pi_j(\kappa) = (0,\theta_{(j+1)j})$ and $\pi_{j+1}(\kappa) = (\theta_{(j+2)(j+1)},\infty).$
\[
 \pi_j(\kappa) = (0,\theta_{(j+1)j}) \quad \mbox{ and } \quad \pi_{j+1}(\kappa) = (\theta_{(j+2)(j+1)},\infty).
\]
Then $\Lambda_{j+1} = L_{(j+1)j}$ so that 
\[
    \cL(\kappa,j+1,-) = \sgn(-\gamma_{j+1}\theta_{(j+2)(j+1)} + L_{(j+1)j}) = -1.
\]
This implies $\Phi_{j+1}(\kappa) = -1$ so that $\kappa$ is not an equilibrium cell. Similarly, if 
%\WILL{$\pi_j(\kappa) = (\theta_{(j+1)j},\infty)$ and $\pi_{j+1}(\kappa) = (0,\theta_{(j+2)(j+1)})$,}
\[
    \pi_j(\kappa) = (\theta_{(j+1)j},\infty) \quad \mbox{ and } \quad \pi_{j+1}(\kappa) = (0,\theta_{(j+2)(j+1)}),
\]
then $\Phi_{j+1}(\kappa) = 1$ and $\kappa$ is not an equilibrium cell. Suppose that $\pi_j(\kappa) = (0,\theta_{(j+1)j})$ for each $j$. Then since $N\dashv 1$ is repressing, 
\[
    \cL(\kappa,1,+) = \sgn(-\gamma_1\theta_{21}+U_{1N}) = 1
\]
and $\Phi_1(\kappa) = 1$. Similarly, $\Phi_1(\kappa) = -1$ if $\pi_j(\kappa) = (\theta_{(j+1)j},\infty)$ for each $j$. Therefore there are no regular equilibrium cells. Since $\tau$ is the only singular loop characteristic cell and equilibrium cells are loop characteristic cells, $\tau$ is the unique equilibrium cell. 
\end{proof}
%%%%%%%%%%%%%%%%%%%%%%%%%%%%%%%%%%%%%%%%%

\subsection{Stability of Equilibria}\label{sec: cfn stability proofs}

We begin by providing the calculation of the characteristic polynomial of the Jacobian $J(x;\eps)$. \red{The Jacobian is given in \eqref{eq: CFS Jacobian}. }

\begin{proof}[Proof of Lemma \ref{lem: cfs det(J)}]
  Let $\til{J} = J(\eps) - \lambda I$. Let $S_N$ be the set of permutations of order $N$. For $\eta\in S_N$ let $\parity(\eta)$ denote the parity of $\eta$, i.e.  $\parity(\eta) = 1$ if $\eta$ is even and $\parity(\eta)=-1$ if $\eta$ is odd. The Liebniz Formula gives 
  \begin{align*}
      \det(\til{J}) = \sum_{\eta\in S_N}\parity(\eta) \prod_{i=1}^N \til{J}_{\eta(i)i}.
  \end{align*}
  The only non-zero derivative of $\Lambda_i$ is the derivative with respect to $x_{i-1}$. Therefore, the only non-zero entries of $\til{J}$ are the diagonal entries $J_{ii} = -\gamma_i-\lambda$ and the entries 
  \[
    \til{J}_{i(i-1)} = \frac{\partial}{\partial x_{i-1}}\Lambda_i = \sigma_{i(i-1)}'. 
  \]
  The only non-zero entries in the sum then correspond to the identity permutation, $\mathsf{id}$ and the permutation $\rho$. Since $\parity(\mathsf{id}) = 1$ and $\parity(\rho) = (-1)^{N-1}$,
  \begin{align*}
      \det(\til{J}) &= \prod_{i=1}^N (-\gamma_j-\lambda) + (-1)^{N-1}\prod_{i=1}^N \sigma_{i(i-1)}'\\
       &= (-1)^N\left(\prod_{i=1}^N (\gamma_i+\lambda) -\prod_{i=1}^N \sigma_{i(i-1)}'\right)\\
       &= (-1)^N\left(\prod_{i=1}^N (\gamma_i+\lambda) -\sgn(\rho)\prod_{i=1}^N M \right)
  \end{align*}
 \end{proof}
 
To determine stability of positive cyclic feedback systems, we apply Descartes' rule of signs and the fact that $J$ is a Metzler matrix (off diagonal entries are non-negative) so that the eigenvalue with largest real part is real (see Theorem 4 of \cite{benvenuti04}). 
\begin{proof}[Proof of Proposition \ref{prop: pos cycle stability}]
   Let $p(\lambda) := (-1)^N\det(J - \lambda I)$ be the characteristic polynomial of $J(x,\eps)$ normalized so that the leading coefficient is positive. Notice that $p(\lambda)$ has all positive coefficients except for possibly the coefficient of $\lambda^0$ which is given by $(-1)^N\det(J) = \prod_j \gamma_j - M(x,\eps)$. If $M(x,\eps) > \prod_j \gamma_j$ then by Descartes' rule of signs, $p$ has a positive real root so that $x$ is unstable. If $M(x,\eps) < \prod_j \gamma_j$, then by Descartes' rule of signs, $M(x,\eps)$ has no positive real roots. Theorem 4 of \cite{benvenuti04} says that $J$ has a real eigenvalue with largest real part. Since $J$ has no positive real roots, this eigenvalue must be negative, implying that $x$ is asymptotically stable. If $M(x,\eps) = \prod_j \gamma_j$ then $\det(J) = 0$ so that $\cS(Z,\eps)$ has a bifurcation at $x$. 
\end{proof}

The stability of equilibria of negative cyclic feedback systems when $N\leq 2$ involves only a simple computation.
\begin{proof}[Proof of Proposition \ref{prop: neg N<3 stable}]
    If $N=1$, then $x=x_1$ and $\dot{x} = -\gamma_1 x + \sigma_{11}(x;\eps)$. Since $\sigma_{11}(x;\eps)$ is non-increasing we have $J(x;\eps) = -\gamma_1 + \sigma_{11}'(x;\eps) \leq -\gamma_1$ so that the equilibrium of $\cS(Z,\eps)$ is stable. If $N=2$, then $\det(J(x;\eps)) = \gamma_1\gamma_2 - \sigma_{21}'(x;\eps)\sigma_{12}'(x;\eps) \geq \gamma_1\gamma_2>0$ since $\sigma_{12}$ is non-decreasing and $\sigma_{12}$ is non-increasing. We also have that the trace of $J(x;\eps)$ is negative. \red{Since $\det(J(x;\eps)>0$ and $\Tr(J)<0$},  the eigenvalues of $J$ have negative real part. 
\end{proof}

\begin{proof}[Proof of Propositions \ref{prop: pos sing cell unstable} and \ref{prop: neg sing cell stability}]
In proving Theorem \ref{thm: FP}, we showed that for $\eps$ small enough the equilibrium $x^\eps$ which converges to $\tau$ satisfies $x^\eps \in \cU_j(\eps)$ where $\cU_j$ is defined as in property (4) of Definition \ref{defn: sigmoid}. Therefore $M(x,\eps) \to \infty$ as $\eps \to 0$. Applying Proposition \ref{prop: pos cycle stability} or \ref{prop: neg cfs stability} as appropriate then proves the propositions. 
\end{proof}

\section{Discussion}\label{sec: discussion}
In this paper we present  explicit and direct correspondence between equilibria of systems of differential equations with sigmoidal nonlinearities and equilibrium objects that are associated to a switching system. ODE models associated to switching systems
are not well defined for points that lie on the family of thresholds associated to these functions. Because of the difficulties that this presents for construction of a well defined flow, we prefer  to think of a switching system not as an ODE model, but as a source of combinatorial (i.e. finite) data that can be used to study sigmoidal systems. Following this philosophy we build upon work of others~\cite{veflingstad07,ironi11} to show that all equilibria and their stability for sufficiently steep  sigmoidal functions can be determined from the data associated to the corresponding switching system. Dissecting further this  rigidity,  the sufficient data consists of network structure and a discrete description of parameter regime in the terms of a set of monotone Boolean functions~\cite{Peter2020}. To facilitate this work, we realize that the dynamics in a neighborhood of loop characteristic cells, that contain so called singular equilibria of the switching system, can be fully understood as a product of cyclic feedback networks. These, in turn, have a simpler structure that can be fully analyzed.

There are several natural extensions of the present work. One set of questions involves asking how far the switching system can be perturbed while maintaining its predictions. Given a switching parameter, how steep must the sigmoidal functions be so that the equilibria given by the switching system data are maintained in the sigmoidal system? How can the switching parameter be chosen so that the equilibria are maintained for the shallowest possible sigmoids? We are currently working on this problem in the context of ramp systems, wherein the sigmoidal functions are replaced by continuous piece-wise linear functions. In this setting, explicit analytic results to these questions can be given.

Another set of questions involves non-stationary dynamics. The switching system dynamics can be represented by a state transition graph that has information not only about the equilibria, but also about recurrent and global dynamics of sigmoidal systems. What is the relationship between recurrent dynamics, say periodic trajectories, in the state transition graph, and periodic orbits and their stability in sigmoidal systems? 
We have already shown that the global dynamics of the STG is closely related to global dynamics of sigmoidal perturbations in two dimensional systems~\cite{us1}. We are currently working on a generalization of this result to higher dimensions. 
We believe that the results of this paper present only a first step in establishing a firm connection between dynamics of sigmoidal models of network dynamics and combinatorial dynamics of state transition graphs of switching systems.

\section*{Acknowledgements}
\red{
HK was supported by JSPS KAKENHI Grant Number 18H03671. HO was supported by JSPS KAKENHI Grant Number 16K03644 and 19K03644, and the Grant for Research of Ryukoku University(Tokubetsu-kenkyu) in the academic year 2018/4-2019/3. TG was partially supported by the National Science foundation under grant DMS-1839299, a DARPA contract FA8750-17-C-0054, and National Institutes of Health award R01 GM126555-01.
KM was partially supported by the National Science Foundation under awards DMS-1839294 and HDR TRIPODS award CCF-1934924, a DARPA contract HR0011-16-2-0033, and National Institutes of Health award R01 GM126555-01. WD was also supported by National Institutes of Health award R01 GM126555-01.
}

\bibliographystyle{siamplain}
 \bibliography{cont_ref}
\end{document}

%% file: KM_definitions_7_2015.tex
\newcommand{\bR}{{\bf R}}
\newcommand{\bS}{{\bf S}}
\newcommand{\bT}{{\bf T}}

\newcommand{\bRN}{{\bf RN}}

\newcommand{\cL}{{\mathcal L}}

\newcommand{\cN}{{\mathcal N}}

\newcommand{\cP}{{\mathcal P}}

\newcommand{\cS}{{\mathcal S}}

\newcommand{\cU}{{\mathcal U}}

\newcommand{\sC}{{\mathsf C}}

\newcommand{\sL}{{\mathsf L}}

\DeclareMathOperator{\sgn}{sgn}

\definecolor{gray85}{gray}{0.85} % 15%
\definecolor{gray8}{gray}{0.8} % 20%
\definecolor{gray7}{gray}{0.7} % 30%
\definecolor{gray6}{gray}{0.6} % 40%
\definecolor{gray5}{gray}{0.5} % 50%
\definecolor{gray4}{gray}{0.4} % 60%
\definecolor{gray35}{gray}{0.35} % 65%

\newcommand{\Sources}{\bS}
\newcommand{\Targets}{\bT}